\documentclass{article}
\usepackage[utf8]{inputenc}     
\usepackage[english]{babel}     
\usepackage[a4paper, left=1.5in, right=1.5in, top=1.5in, bottom=1.5in]{geometry}
\usepackage{graphicx}           
\usepackage{amsmath,amssymb}    
\usepackage{amsthm}             
\usepackage{mathtools}          
\usepackage{enumitem}           
\usepackage{listings}           
\usepackage{todonotes}          
\usepackage{hyperref}           
\usepackage{cleveref}
\usepackage{physics}
\usepackage{dsfont}

\usepackage{pgfplots}
\usepackage{ulem}
\pgfplotsset{compat=1.15}
\usepackage{mathrsfs}
\usepackage{caption}
\usepackage{subcaption}
\usepackage{float}
\usepackage{mwe}
\usepackage{cancel}
\usepackage{verbatim}
\usepackage{bm} 
\usepackage{placeins} 
\usepackage{algorithm} 
\usepackage{algpseudocode}
\usepackage{pdfpages}
\floatname{algorithm}{Algorithm}

\usepackage[noblocks]{authblk}

\usepackage[nottoc]{tocbibind}
\DeclarePairedDelimiterX\set[1]\lbrace\rbrace{#1}

\newtheorem*{theorem*}{Theorem}
\newtheorem{theorem}{Theorem}[section]

\newtheorem{lemma}[theorem]{Lemma}
\newtheorem{proposition}[theorem]{Proposition}

\theoremstyle{definition}
\newtheorem{definition}[theorem]{Definition}

\newtheorem*{problem*}{Problem}

\newtheorem{remark}{Remark}[theorem]

\def \a  {\alpha}

\def \g  {\gamma}

\def \d  {\delta}
\def \D  {\Delta}
\def \e  {\varepsilon}

\def \l  {\lambda}

\def \Om {\Omega}

\def \r  {\rho}

\def \x {{\bm{x}}}

\def \del {\nabla}

\def \p  {\partial}
\def \R  {\mathds{R}}
\def \N  {\mathbb{N}}

\def \Maxi {\mathrm{M}}
\def \Mini {\mathrm{m}}

\numberwithin{equation}{section}

\newcounter{AssBio}
\stepcounter{AssBio}

\newcounter{AssEx}
\stepcounter{AssEx}

\title{Robust and fast iterative method for the elliptic Monge-Ampère equation}
\author[1]{R.N. Köhle\thanks{email: \href{mailto:r.n.kohle@tue.nl}{r.n.kohle@tue.nl}}}
\author[2]{K.T.W. Menting}
\author[1]{K. Mitra}
\author[1]{J.H.M. ten Thije Boonkkamp}
\affil[1]{Department of Mathematics and Computer Science, Eindhoven University of Technology, P.O. Box 513, 5600 MB Eindhoven, The Netherlands}
\affil[2]{Infiniot, High Tech Campus 10, 5656 AE Eindhoven, The Netherlands}

\date{\today}

\begin{document}

\maketitle

\begin{abstract}
    This paper introduces a fast and robust iterative scheme for the elliptic Monge-Ampère equation with Dirichlet boundary conditions. The Monge-Ampère equation is a nonlinear and degenerate equation, with applications in optimal transport, geometric optics, and differential geometry. The proposed method linearises the equation and uses a fixed-point iteration (L-scheme), solving a Poisson problem in each step with a weighted residual as the right-hand side. This algorithm is robust against discretisation, nonlinearities, and degeneracies. For a weight greater than the largest eigenvalue of the Hessian, contraction in $H^2$ and $L^\infty$ is proven for both classical and generalised solutions, respectively. The method's performance can be enhanced by using preconditioners or Green’s functions. Test cases demonstrate that the scheme outperforms Newton's method in speed and stability.

\end{abstract}



\section{Introduction} \label{sec: 1}
The Monge-Ampère equation is a nonlinear, second-order partial differential equation which plays a significant role in various areas of mathematics and physics including differential geometry, optimal transport, and geometric optics \cite{de2014monge,glimm2003optical,ten2025inverse,trudinger2008monge} .
Let $\Om \subset \R^d$ be open and bounded for $d \in \N$. The $d$-dimensional Monge-Ampère equation on $\Om$ is a boundary value problem of the form
\begin{align}\label{eq:main_problem}
  \left\{
  \begin{aligned}
    \det(\mathrm{D}^2 u) &= f(\cdot,\nabla u)   &&\quad \text{in } \Omega, \\
    u &= \gamma                                 &&\quad \text{on } \partial\Omega,
  \end{aligned}
  \right.
\end{align}
where $u \in C^2(\bar{\Om})$ and $\mathrm{D}^2u$ is the Hessian matrix of $u$.

The Monge-Ampère equation with Dirichlet boundary condition appears in the context of a hypersurface described by $u:\Om\to \R$ having fixed boundary values $\gamma$ on $\p\Om$, and prescribed Gaussian curvature $K(\bm{x})$ for all $\bm{x}\in \Om$. This leads to the equation
\begin{equation}\label{eq:guass_curvature}
    \frac{\det (\mathrm{D}^2 u)}{(1+\vert \nabla u \vert ^2)^{(d+2)/2}}=K,
\end{equation}
 corresponding to \eqref{eq:main_problem} with $f(\bm{x}) = K(\bm{x})(1+\vert \nabla u(\bm{x}) \vert ^2)^{(d+2)/2}$ \cite{de2014monge}.
The Monge-Ampère equation is also an essential tool in the field of optimal transport, describing the problem of moving one mass distribution (in $\Om$) to another (say in $\vartheta$) while minimising a given (quadratic) cost function \cite{santambrogio2015optimal}.
This cost function typically represents the Euclidean distance each mass element needs to be moved. 
Several problems in inverse optical design can be framed as optimal transport problems. 
In these cases, the `mass' corresponds to energy density, and the cost function corresponds to the optical path length of light rays \cite{Prins,yadav2019least}.
However, in the optimal transport formulation of optics, we need the transport boundary condition $\nabla u(\p \Om)=\p \vartheta$ which introduces further complications. For simplicity, in this paper, we focus on Dirichlet boundary conditions.

The Monge-Ampère equation is nonlinear and may change its type from elliptic to parabolic to hyperbolic depending on the sign of $f$ \cite{de2014monge,bertens2021numerical}, thus, requiring different approaches to solve. This is called `degeneracy' and correspondingly the surface described by $u$ is convex/concave, flat (0 Gaussian curvature), or saddle shaped. Nonlinearity and degeneracy of the problem make it quite challenging to solve system \eqref{eq:main_problem} numerically. 
Despite this, several distinct approaches have been developed. 
Some studies focus on Dirichlet boundary conditions, while others employ transport boundary conditions for the elliptic Monge-Ampère equation.
For solving the Monge-Ampère equation with Dirichlet boundary conditions using finite elements, a time-marching scheme for the resulting nonlinear system of equations is outlined in \cite{Awanou}, where a pseudo-time parameter is introduced to update the solution iteratively. This scheme is applicable to Alexandrov solutions with both finite difference or finite element discretisations \cite{awanou2017standard, awanou2016standard}. Another finite element discretisation is proposed in \cite{glowinski2019finite}, where the Monge-Ampère equation is solved in a two-stage method. The authors also use a pseudo-time parameter in a Newton-like scheme. A vanishing moments approach is described in \cite{brenner2012finite,vanishing_moments}, which introduces a regularization that approximates the second-order Monge-Ampère equation using a fourth-order quasilinear PDE.
An alternative approach, presented in \cite{Froese}, uses a wide-stencil scheme to solve the same problem, employing a damped Newton method to obtain the solution. A finite element method for a regularised formulation of the Monge-Ampère equation in two dimensions has been considered in \cite{gallistl2023convergence}, and an adaptive discretisation algorithm presented in \cite{gallistl2024stability}.
Additionally, the Dirichlet problem can be addressed using neural networks, as demonstrated in \cite{NYSTROM}. 

With transport boundary conditions, one approach to discretise the Monge-Am\-père equation is the use of finite element methods, see \cite{kawecki2018} for a non-variational version with oblique boundary conditions.  To linearise the problem, a Newton-Raphson iteration is employed, which transforms the problem into a sequence of elliptic equations.
A projection method to solve these equations is described in \cite{froese_Transp}, where in each iteration the current guess is projected onto the space of convex functions. The resulting system is then solved using Newton iteration.
In \cite{Prins} an iterative least-squares solver is proposed for optical applications, and generalised for non-quadratic cost functions in \cite{yadav2019least, romijn2020monge, ten2025inverse}.
Alternatively, a novel artificial neural network-based approach to solve this problem is presented in \cite{hacking2024}. 

For the hyperbolic Monge-Ampère equation ($f<0$), the methods of characteristics is used in \cite{bertens2021numerical} to transform the partial differential equation into two coupled systems of ordinary differential equations. These can be solved with explicit one-step methods. Alternatively, a least-squares solver for the hyperbolic Monge-Ampère equation with transport boundaries is described in \cite{bertens2023iterative}.

Existing iterative algorithms, as demonstrated through numerical examples, face significant challenges in terms of convergence. Newton's method and other iterative solvers are inherently reliant on an appropriate initial guess and mesh size. Moreover, the convergence of these iterations is often constrained by the nonlinearity and degeneracy of the problem, which can result in divergence. 
To overcome these issues,  Awanou in 2015 \cite{Awanou} proposed an iterative scheme (the so-called 'pseudo time-marching') which guarantees linear convergence of iterates in $H^1$ for classical solutions. Independently, inspired by robust linearisation schemes such as the  $L$-scheme \cite{POP2004365,MITRA20191722,stokke2023adaptive,javed2025robust}, and motivated to solve a fixed problem in each iterative step, we arrived at the same scheme: for a current iterate $u^{i}\in C^2(\bar{\Om})$, it computes the next  iterate $u^{i+1}$ from a Poisson problem, which will be described later. We summarise the main achievements of this paper below.

\textbf{Main result:} 
The iterations solve a Poisson equation at each step with a weighted residual in the right hand side. Thus, for an arbitrary initial guess $u^0\in C^2(\bar{\Om})$, the iterations are always well-posed, as opposed to the Newton scheme which would require convexity of each $u^i$ (see \Cref{sec:ProbForm}). Next, under convexity conditions on the iterates, we prove that the scheme converges linearly in $H^2$ for classical solutions (stronger convergence than in \cite{Awanou}), and in $L^\infty$ for viscosity solutions (weaker notion of solutions) even when $f$ depends on $\nabla u$, irrespective of the spatial discretization (differentiating it from \cite{Awanou}). Going beyond, since in each iteration we essentially solve a Poisson problem, numerous acceleration techniques designed for the Poisson equation become available. We explore two such options, Green's function and preconditioners, which can be computed once, and used in every iteration to accelerate the solution process. We show that among the standard preconditioners, preconditioned algebraic multigrid gives the best performance. Numerical results demonstrate that the scheme is extremely robust, converging for all mesh sizes, with vastly different (even saddle shaped) initial guesses, and rapidly oscillating $f$ which can even approach the degeneracy limit ($f=0$). Newton's convergence is shown to be limited in terms of all these aspects. Furthermore, due to the acceleration possible, we show that in terms of CPU time, the method outperforms Newton in all cases when both methods converge, despite Newton requiring less iterations for coarser meshes. This establishes our method as a fast and robust way to solve the Monge-Ampère equation.

\section{Mathematical preliminaries}\label{sec:ProbForm}
In the following, we will study the $d$-dimen\-sional elliptic Monge-Ampère equation on a domain $\Om$ with Dirichlet boundary condition. 
We assume that $\Om \subset \R^d$ is open and bounded with a Lipschitz boundary $\partial \Om$. 
Higher regularity of the boundary will be assumed when considering classical solutions.

\textbf{Functional spaces:}
In this work, $C(\Om)$ denotes the set of continuous functions on $\Om$, $C^k(\Om)$  the set of functions that have continuous derivatives up to the $k^{\rm th}$ order ($k\in \N$), and $C^{k,\a}(\Om)$ the set of functions having up to $k^{\rm th}$ order derivatives that are H\"older continuous with exponent $\alpha$ (where $\a\in (0,1)$) (see \cite[Chapter 5]{evans2022partial}). Continuous and Lipschitz continuous functions in $\Om$ will be associated with $C^{0,0}(\Om)$ and $C^{0,1}(\Om)$; respectively. Furthermore, 
$L^2(\Om)$  denotes the set of measurable functions that are square integrable, $L^\infty(\Om)$ consists of all measurable functions $u$ that are essentially bounded, and the Sobolev space $H^k(\Om)$ is the subset of $L^2(\Om)$ with weak derivatives up to the $k^{\rm th}$ order in $L^2(\Om)$. The norm of a space $\mathcal{V}$ will be denoted by $\|\cdot\|_{\mathcal{V}}$.

\textbf{Matrix relations:}  The $d\times d$ identity matrix is $\mathbb{I}_d$. Subsequently, $\bm{A},\bm{B} \in \R^{d \times d}$ will denote symmetric matrices. This implies $\bm{A}$ has real eigenvalues, say $\l_{A,1}\leq \dots\leq \l_{A,j}\leq \dots \leq \l_{A,d}$, and for some orthogonal matrix $\bm{Q}_A$, it holds:
\begin{equation}
    \bm{A}= \bm{Q}_A^{\rm T} {\rm diag}(\l_{A,j}) \bm{Q}_A.
\end{equation}
We have the matrix ordering $\bm{A} \prec \bm{B}$, if $\bm{B} - \bm{A}$ is positive definite, and  $\bm{A} \preceq \bm{B}$ if $\bm{B} - \bm{A}$ is positive semi-definite,  $\bm{A} \succeq \bm{B}$, if $\bm{A} -\bm{B}$ is negative semidefinite and $\bm{A} \succ \bm{B}$, if $\bm{B} - \bm{A}$ is negative definite. The Frobenius inner-product of $\bm{A}=(a_{jk})_{1\leq j,k\leq d}$ and $\bm{B}=(b_{jk})_{1\leq j,k\leq d}$ is defined as
    \begin{equation}\label{eq:frob_prod}
        \bm{A : B} := \text{tr}(\bm{A}^\mathrm{T} \bm{B}) = \text{tr}(\bm{B}^\mathrm{T} \bm{A})  = \sum_{j,k = 1}^d a_{kj} b_{kj}.
    \end{equation}  
We will use \textit{Jacobi's formula}: Let $\bm{A}=\bm{A}(t) \in \R^{d \times d}$ for $t \in \R$. We assume that $\bm{A}(t)$ is invertible for all $t$. Then, we have
\begin{equation}\label{eq:jacobi}
    \frac{\mathrm{d}}{\mathrm{d}t} \det(\bm{A}(t)) = \text{cof}(\bm{A}(t))\bm{:} \bm{A}'(t),
\end{equation}
where $\text{cof}(\bm{A})= \det(\bm{A})(\bm{A}^{-1} )^{\rm T}$ is the co-factor matrix. An immediate consequence of the mean value theorem is 
\begin{equation}\label{eq:DetMeanValue}
    \det (\bm{A})-\det (\bm{B})= \text{cof}(t\bm{A} + (1-t)\bm{B})\bm{:} (\bm{A}-\bm{B}),
\end{equation}
for some $t\in [0,1]$.

\subsection{Assumptions on data}
We assume the following properties of the data:
\begin{enumerate}[label=(A\arabic*)]
\item \label{ass:Afa}
The right-hand side $f(\bm{x,y})$ of \eqref{eq:main_problem}  is $C^{0,\a}(\Om)$ with respect to $\bm{x} \in \Om$ for any fixed $\bm{y}\in \R^{d}$ for some $0\leq \a\leq 1$. It is Lipschitz continuous with respect to $\bm{y}$ for a fixed $\bm{x} \in \Om$, i.e., there exists a Lipschitz constant $\mu_f > 0$ such that for all $\bm{x} \in \Om$ and $\bm{y}_1, \bm{y}_2 \in \R^d$,
\begin{equation}\label{eq:lipschitz_of_f}
    \vert f(\bm{x},\bm{y}_1) -f(\bm{x},\bm{y}_2) \vert \leq \mu_f \vert \bm{y}_1 - \bm{y}_2 \vert.
\end{equation}
Moreover, $f(\bm{x,y})$ is positive and bounded: there exist constants $f_{\Mini}$ and $f_{\Maxi}$ such that for all $\bm{x}\in\Om$ and for all $\bm{y}\in\R^d$,
\begin{equation}\label{eq:boundedness}
    0 \leq  f_{\Mini} \leq f(\bm{x,y}) \leq f_{\Maxi}.
\end{equation}
\item \label{ass:Aga}
The Dirichlet boundary value function $\gamma:\p\Om\to \R$ is the trace of a convex (or concave if $d$ is even) $C^{2,\a}(\bar{\Om})$ function which will also be denoted by $\gamma$.
\end{enumerate}

\begin{remark}[Smallness assumption on $\mu_f$]
    In our analysis we require $\mu_f$ in \eqref{eq:lipschitz_of_f} to be small, see \Cref{thm:convergence}. The case $\mu_f=0$ models the problem of a uniform target distribution in optimal transport, and is commonly discussed in the literature \cite{gutierrez2001monge}.  
\end{remark}

\begin{remark}[Choice of  $\a$]
When discussing classical solutions of \eqref{eq:main_problem} (in $C^2(\Om)$), we would consider $\a>0$. But for the generalised solutions (see \Cref{def:gen_sol}), we can take $\a=0$, and even $f(\bm{\cdot},\bm{y})$ discontinuous in $\Om$.  
\end{remark}

\subsection{Solution concepts and well-posedness}

 \begin{proposition}[Existence, uniqueness, and regularity of classical solutions] Let $\p\Om$ be in $C^{3,1}$, $\gamma\in C^{3,1}(\p\Om)$ and assume \ref{ass:Afa}--\ref{ass:Aga}. 
    Then there exists a unique convex $u\in C^{1,1}(\bar{\Om})$ that solves \eqref{eq:main_problem} almost everywhere. Moreover, if $f_\Mini>0$, then there exists  $\l_\Mini,\,\l_\Maxi:\Om\to (0,\infty)$ such that for almost every $\bm{x}\in \Om$
    \begin{equation}\label{eq:lmlM}
       \bm{0}\prec \l_{\Mini}(\bm{x}) \mathbb{I}_d \preceq \mathrm{D}^2 u(\bm{x})  \preceq \l_{\Maxi}(\bm{x}) \mathbb{I}_d.
       \end{equation}\label{prop:exist_class}
 \end{proposition}

 The proof of this statement can be found in \cite{guan1998dirichlet}.
A solution $u\in C^2(\Om)$ (or $\in C^{1,1}(\Om)$ as above) is called a \emph{classical solution} of \eqref{eq:main_problem}.
However, there are much weaker concepts of solutions known for the Monge-Ampère equation that require lower regularity of $\p\Om$, $\gamma$, and $f$. For this, we set $\mu_f=0$ in \eqref{eq:lipschitz_of_f} for simplicity. For a convex $v\in C(\Om)$, the sub-differential $\p v(\bm{x})\subset \R^d$ at $\bm{x}\in \Om$ is defined as:
 \begin{equation*}
 \p v(\bm{x}):= \{\bm{p}\in \R^d\;|\; v(\bm{x})+\bm{p}\bm{\cdot} (\bm{y}-\bm{x})\leq v(\bm{y}),\;\; \forall\; \bm{y}\in \Om\}.
 \end{equation*}
 The Monge-Ampère measure $\mathcal{M}v$ is then defined as
\begin{equation}\label{eq:monge-measure}
\mathcal{M} v(\vartheta):= \nu\left (\cup_{\bm{x}\in \vartheta} \p v(\bm{x})\right ) \text{ for all Borel sets } \vartheta\subseteq \Om,
\end{equation}
where $\nu$ is the Lebesgue measure in $\R^d$. Observe that if $u\in C^2(\Om)$ then for all measurable $\vartheta\subset \Om$\, one has $\mathcal{M} v(\vartheta)=\int_\vartheta\det (\mathrm{D}^2 v(\bm{x}))\mathrm{d}x$. We define
 \begin{definition}[Generalised/Alexandrov solution of \eqref{eq:main_problem}]\label{def:gen_sol}
Let $\mu_f=0$ in \ref{ass:Afa}. The generalised solution $u\in C(\bar{\Om})$ of \eqref{eq:main_problem} is convex and satisfies $\mathcal{M} u =f$ and $u=\gamma$ on $\p\Om$.
 \end{definition}
 
 Following \cite{gutierrez2001monge,gallistl2023convergence} we conclude that:
\begin{proposition}[Existence, uniqueness, and regularity of generalised solutions] \label{prop:gen_sol}
For $\Om$ convex and Lipschitz, $f$ bounded and positive, $\mu_f=0$ in \eqref{eq:lipschitz_of_f}, and $\g\in C(\bar{\Om})$, the generalised solution $u\in C(\bar{\Om})$ exists and is unique. Moreover, if $f$ is $C^{0,\a}(\Om)$ as in \ref{ass:Afa}, $0<f_\Mini<f_\Maxi<\infty$, and $\gamma\in C^{2,\a}(\bar{\Om})$ as in \ref{ass:Aga}, then additionally $u\in C^{2,\alpha}_{\rm loc}(\Om)$, and for every compact subset $\vartheta\Subset \Om$ there exists $\l_\Mini,\, \l_\Maxi: \vartheta\to (0,\infty)$ such that \eqref{eq:lmlM} holds in $\vartheta$.
\end{proposition} 
 
 \begin{remark}[Convex and concave solutions]
Above we have only discussed convex solutions to the Monge-Ampère equation. It is clear that in the case when $d$ is even, $\mu_f=0$, and $\g=0$, that there exists also concave solutions to \eqref{eq:main_problem} obtained simply by taking the negative of the convex solution. Our iterative method is well adopted for this solution as well; see \Cref{rem:concave}.  For simplicity, we will mainly discuss the convex case.
\end{remark} 

\section{Iterative linearisation}\label{sec:iterarive}
In this section, we will discuss iterative linearisation schemes for the Monge-Ampère equation. For a function $v\in C^2(\Om)$, we introduce the \emph{residual} mapping $\r(v)\in C(\Om)$ as
\begin{equation}\label{eq:residual}
  \r(v) := \det(\mathrm{D}^2 v)-f(\cdot,\del v).
\end{equation}
For $i \in \N$, let $u^i \in C^2(\Om)$ be a given approximation of $u$. The objective is to find an update $v^{i+1} \in C^2(\Om) $ satisfying 
\begin{equation*}
  \r(u^i+v^{i+1}) \cong 0 \;\; \text{ and update, },\;\; u^{i+1}:=u^i+v^{i+1}.
\end{equation*}
We introduce an auxiliary function $w$ and scalar $t\geq 0$,  satisfying $v^{i+1}:= tw$. Then, 
using a Taylor series of $\r(u^i+ tw)$ around $t =0$, we have
 \begin{equation}\label{eq:g_Taylor}
 \r(u^i+v^{i+1})=\r(u^i+tw)= \r(u^i) + \frac{\mathrm{d}}{\mathrm{d}t}(\r(u^i+t w))\Bigg|_{t=0}t + \mathcal{O}(t^2)= 0.
 \end{equation}
Recalling Jacobi's formula \eqref{eq:jacobi}, we obtain
\begin{equation*}
    \frac{\mathrm{d}}{\mathrm{d}t}\det (\mathrm{D}^2u^i+ t\mathrm{D}^2w)= \text{cof}(\mathrm{D}^2u^i+t\mathrm{D}^2w)\bm{:}\mathrm{D}^2w.
\end{equation*}
The chain rule further gives 
\begin{equation*}
    \frac{\mathrm{d}}{\mathrm{d}t} f(\cdot,\nabla( u^i + tw))= \del_{\bm{y}} f(\cdot, \nabla( u^i + tw)) \bm{\cdot} \nabla w.
\end{equation*}
Thus, we can express the second term in the right-hand side of \eqref{eq:g_Taylor} as
\begin{equation*}
  \frac{\mathrm{d}}{\mathrm{d}t}(\r(u^i+t w))\Bigg|_{t=0} t=\text{cof}(\mathrm{D}^2u^i)\bm{:}\mathrm{D}^2(tw) -\del_{\bm{y}}f(\cdot,\del u^i)\bm{\cdot} \del(t w).
\end{equation*}
Introducing the symmetric matrix and the vector field
\begin{equation}\label{eq:g_jac_for}
\bm{C}^i:= \text{cof}(\mathrm{D}^2u^i), \quad  \bm{q}^i:= \del_{\bm{y}}f(\cdot,\del u^i),
\end{equation}
respectively, recalling $v^{i+1}=tw$, and ignoring the higher order $\mathcal{O}(t^2)$ terms, we obtain a linear approximation for $v^{i+1}$, i.e. 
\begin{equation}\label{eq:cof_rho}
  \bm{C}^i\bm{:} \mathrm{D}^2v^{i+1} -\bm{q}^i\bm{\cdot} \del v^{i+1} = -\r(u^i).
\end{equation}

\subsubsection{Newton iteration}
For Newton iteration we solve the following system
\begin{equation}\label{eq:poisson_itern}
  \left\{
  \begin{aligned}
    \bm{C}^i\bm{:} \mathrm{D}^2 v^{i+1} - \bm{q}^i \bm{\cdot} \del v^{i+1} &= -\r(u^i),         &&\quad \text{in } \Omega, \\
    v^{i+1} &= \gamma - u^i  ,                        &&\quad \text{on } \partial\Omega.
  \end{aligned}
  \right.
\end{equation}
The iterations are well-defined by the Lax-Milgram Theorem \cite{evans2022partial} if $\bm{C}^i$ is either positive or negative symmetric definite, which is the case when $u^i$ is strictly convex or concave, respectively, and the advection term $\bm{q}^i$, bounded in $L^\infty$ above by $\mu_f$ defined in \eqref{eq:lipschitz_of_f}, is small. However, if $\bm{C}^i$ is neither positive nor negative definite, then the existence of solutions of the boundary value problem \eqref{eq:poisson_itern} is not guaranteed, since it violates the coercivity condition of Lax-Milgram.
Thus, the Newton iteration requires convexity/concavity of iterates for well-posedness, which can be hard to enforce. This leads to instability and strong sensitivity for the initial guess as we will see in \Cref{sec:numRes}. Consequently, some remediation is needed, for which we propose the following.
 

\subsubsection{L-Scheme}\label{subsec:L-scheme}
We take inspiration from \cite{POP2004365,MITRA20191722, stokke2023adaptive, javed2025robust} where it was shown that replacing $u^i$-dependent coefficients by carefully selected constants improved the convergence properties of iterations for nonlinear elliptic problems. Here, we replace the cofactor matrix dependent on $u^i$ by the identity matrix scaled by a global lumped constant $\Lambda^i \in \R\setminus\{0\}$ which may depend on $u^i$, i.e., $\bm{C}^i\mapsto \Lambda^i \mathbb{I}_d$.
Moreover, we set the advection-like term $\bm{q}^i\equiv \bm{0}$ for stability. Observing that $\tr(\Lambda^i \mathbb{I}_d\mathrm{D}^2 v)=\Lambda^i \tr( \mathrm{D}^2 v) =\Lambda^i \D v$, we get the Dirichlet problem for the update $v^{i+1}$:
\begin{subequations}\label{eq:L-scheme}
\begin{equation}\label{eq:poisson_iter}
\left\{
\begin{split}
       \Lambda^i \D v^{i+1} &= -\r(u^i),  &&\text{in} \quad \Om,\\
    v^{i+1} &=  \gamma -u^i, &&\text{on} \quad \p \Om,
\end{split}
\right.
\end{equation}
with $u^{i+1}:\Om \to \R$ being computed afterwards through
\begin{equation}
  u^{i+1} := u^i + v^{i+1} .
\end{equation} 
\end{subequations}
The above scheme will henceforth be referred to as the \textbf{L-scheme}. We show below that it has global well-posedness and consistency properties as opposed to Newton.

\begin{theorem}[Well-posedness and consistency of the L-scheme]\label{thm:exists}
    Let $\p\Om$ be $C^{2,\a}$, and \ref{ass:Afa}--\ref{ass:Aga} hold with $\a\in [0,1]$. Let the iteration index be denoted by $i\in \N$, the initial condition $u^0\in C^{1,1}(\bar{\Om})$, and $\{\Lambda^i\}_{i\in \N}\subset [\Lambda_\Mini,\Lambda_\Maxi]$ be a sequence of constants for fixed $0<\Lambda_\Mini\leq \Lambda_\Maxi<\infty$. Then, there exists a unique sequence $\{u^{i}\}_{i\in \N} \subset  C^{1,1}(\bar{\Om})$ satisfying \eqref{eq:L-scheme} almost everywhere for each $i\in \N$. Additionally, if $\a>0$ and $u^0\in C^{2,\a}(\Om)$, then $u^i\in C^{2,\a}(\Om)$ for all $i\in \N$. Furthermore, the scheme is consistent in the sense that the update $v^{i+1}=u^{i+1}-u^{i}= 0$, in $\bar{\Om}$ if and only if $u^i=u$ where $u\in C^{1,1}(\bar{\Om})$ is the solution of \eqref{eq:main_problem} described in \Cref{prop:exist_class}. 
\end{theorem}

\begin{proof}
We prove the theorem by induction. Let $u^i\in C^{1,1}(\bar{\Om})$ for some $i\in \N$. 
    Due to Rademacher's theorem \cite[Chapter 5.8]{evans2022partial}, we get that ${\rm D}^2 u^i$ exists almost everywhere in $\Om$, and is essentially bounded. Thus, $\r(u^i)\in L^\infty(\Om)$. Then, by \cite[Theorem 8.34]{gilbarg1977elliptic}, there exists a unique solution $v^{i+1}\in C^{1,1}(\bar{\Om})$ of \eqref{eq:poisson_iter}, and thus $u^{i+1}\in C^{1,1}(\bar{\Om})$ exists uniquely.

If in addition $u^i\in C^{2,\a}(\Om)$, then ${\rm D}^2 u^i\in C^{0,\a}(\Om)$. Since $f(\cdot,\bm{y})$ is $C^{0,\a}(\Om)$ from \ref{ass:Afa}, this implies $f(\bm{\cdot},\del u^i(\bm{\cdot}))$ is also  
$C^{0,\a}(\Om)$, as $f(\bm{x},\cdot)$ and $\nabla u^i$ are both Lipschitz. This implies $\r(u^i)\in C^{0,\a}(\Om)$. Applying \cite[Theorem 6.24]{gilbarg1977elliptic} we get that $v^{i+1}\in C^{2,\a}(\Om)$ and consequently $u^{i+1}\in C^{2,\a}(\Om)$.

Finally, $v^{i+1}=0$ implies $\r(u^i)=0$ in $\Om$, and $u^i=\gamma$ on $\p\Om$. The uniqueness of solution from \Cref{prop:exist_class} implies $u^i=u$. The equivalence in the opposite direction is trivial.
\end{proof}
 In subsequent sections we will show analytically and numerically that if $\Lambda^i$ is large enough (specifically $\Lambda^i\geq \l_\Maxi^{d-1}$), then the L-scheme \emph{converges linearly} under convexity of the iterates. Furthermore, the L-scheme \emph{solves a Poisson problem} $\D v^{i+1}=-\rho(u^i)/\Lambda^i$ in each iteration which makes it \emph{amenable to the application of different acceleration techniques} like highly efficient preconditioners, Green's functions, model order reduction techniques, etc. This will be explored in \Cref{sec:two-strategies}.

\begin{remark}[Modification for the concave case]\label{rem:concave}
    For obtaining a concave solution to \eqref{eq:main_problem} when it is possible, e.g., when $d$ is even, we simply need to set $\Lambda^i$ as negative constants instead of a positive constant. The convergence analysis is identical to what follows.
\end{remark}

\section{Convergence analysis}
In the following section, we are going to prove the convergence of the L-Scheme. In \Cref{sec:classical}, we show the convergence when a classical solution of \eqref{eq:main_problem} exists. Additionally, in \Cref{sec:generalised}, we prove the convergence for a generalised solution.
To prove convergence, we introduce the error $e^i$, defined for all $i \in \N$ as 
\begin{equation*}
  e^i := u^i - u,
\end{equation*}
where $u^i$ is the $i$-th iterate, and $u$ the solution of \eqref{eq:main_problem}. 

\subsection{Convergence of classical solutions}\label{sec:classical}

\begin{theorem}[Convergence of the L-scheme for classical solutions]\label{thm:convergence} 
Let \ref{ass:Afa}--\ref{ass:Aga} hold, and $\p\Om$ be a $C^2$-boundary.
Consider $u\in C^{1,1}(\Om)$ (alternately $u\in C^{2}(\bar{\Om})$) as the convex classical solution of \eqref{eq:main_problem} satisfying the condition \eqref{eq:lmlM} for some $\l_\Mini>0$. 
Let $\{u^{i}\}_{i\in \N}\in C^{1,1}(\bar{\Om})$ denote the L-scheme iterates generated by solving \eqref{eq:L-scheme}. 
Let $\l^i_{\Maxi}(\bm{x})>0$ be the maximum eigenvalue of $\mathrm{D}^2u^i(\bm{x})$. We fix an iteration index $i\in \N$, and let $u^i=\gamma$ on $\p\Om$. 
Choose a lumped constant $\Lambda^i$ satisfying
  \begin{equation}\label{eq:lumbed_constant}
    \Lambda^i \geq \|\l_{\Maxi}^i(\bm{x})\|^{d-1}_{L^\infty(\Om)} .
  \end{equation}
Assume that the error $e^i=u^i-u$ is convex in $\Om$. If $\mu_f$ is smaller than a positive constant dependent only on $\Om$ and $f(\cdot,\bm{0})$, then
\begin{equation}\label{eq:contraction_L2}
\Vert \D e^{i+1}\Vert_{L^2(\Om)} \leq q_i \Vert \D e^i \Vert_{L^2(\Om)} \;\; \text{ for a contraction rate } q_i\in (0,1).
\end{equation}
Moreover, if $\mu_f= 0$, and $e^i$ are convex for all $i\in \N$, then
\begin{enumerate}[label=(\roman*)]
  \item The choice $\Lambda^i = \|\D u^0-(d-1)\l_\Mini\|_{L^\infty(\Om)}^{d-1}$ satisfies \eqref{eq:lumbed_constant} for all $ i \geq 0$.
  \item If for all $i \geq 0$, the $\Lambda^i$ satisfying \eqref{eq:lumbed_constant} are bounded from above by some $\bar{\Lambda} > 0$, e.g., for the choice of constant $\Lambda^i$ in (i), then the fixed-point method converges in $C^{1,1}(\bar{\Om})$ and linearly in $H^2(\Om)$ with contraction rate $q:= 1 - \left(f_\Mini/\bar{\Lambda}^{\frac{d}{d-1}}\right)$.
\end{enumerate}
\end{theorem}

The proof of \Cref{thm:convergence} relies heavily on the following inequality:

\begin{lemma}[An important inequality]\label{lem:Classical}
    Under the assumptions and definitions of \Cref{thm:convergence} we have the following estimate almost everywhere in $\Om$:
\begin{align}
     0 \leq \left(1 - \frac{(\l_{\Maxi}^i)^{d-1}}{\Lambda^i}\right) \D e^i \leq \D e^{i + 1} + \frac{e^i_f}{\Lambda^i} \leq \left(1 - \frac{1}{\Lambda^i}\max\left\{\frac{f_\Mini}{\l_\Maxi^i},\l_\Mini^{d-1} \right\}\right)\D e^i.\label{eq:to_proof_thm}
\end{align}
where $e^i_f:=f(\cdot,\del u^i) - f(\cdot,\del u)$, 
\end{lemma}

\begin{remark}[Convexity assumption on the errors] The major assumption in \Cref{thm:convergence} is that the errors $e^i$ are convex functions for all $i\in \N$. Proving this statement can be quite technical and beyond the scope of this work, see \cite{korevaar1983capillary} for some requirements on convexity of the solution for elliptic problems. For $\mu_f=0$  however, the convexity assumption is consistent in the sense that convexity of $e^i$ implies $\D e^{i+1}\geq 0$ from \eqref{eq:to_proof_thm}, which holds when $e^{i+1}$ is convex. Thus,  there is no logical contradiction in all the iterates being convex.
\end{remark}

    Since $u$ is strictly convex, $\mathrm{D}^2 u$ is positive definite. For later use, we introduce the matrix $\bm{A}^i \in \R^{d \times d}$ defined for some $t\in [0,1]$ as
\begin{equation*}
  \bm{A}^i := t\mathrm{D}^2u^i + (1-t)\mathrm{D}^2u.
\end{equation*}
 First, we look at the implications of the assumption that $u^i-u$ is convex. 
It means that 
    \begin{align}\label{eq:D2u_bounds}
        0\prec \l_\Mini \mathbb{I}_d\overset{\eqref{eq:lmlM}}\preceq \mathrm{D}^2 u \preceq \bm{A}^i \preceq \mathrm{D}^2 u^i\preceq \l^i_\Maxi \mathbb{I}_d,
    \end{align}
almost everywhere in $\Om$.    Consequently, if $\l_j, \l^i_j \in [\l_\Mini,\l_\Maxi^i]$ are the eigenvalues of $\mathrm{D}^2 u$ and $\bm{A}^i$, respectively, for $j = 1,2,\ldots, d$, then we have
\begin{align*}
        &0\prec \frac{1}{\l_\Maxi^i} \mathbb{I}_d\preceq (\bm{A}^i)^{-1} \preceq (\mathrm{D}^2 u)^{-1}\preceq \frac{1}{\l_\Mini} \mathbb{I}_d,\quad (\bm{A}^i)^{-1} \leq \frac{1}{\min \l_j^i}\mathbb{I}_d,\\
        &f_\Mini\leq f\overset{\eqref{eq:main_problem}}{=}\Pi_{j=1}^d \l_j =\det(\mathrm{D}^2 u)\leq \det(\bm{A}^i)=\Pi_{j=1}^d \l^i_j\leq \left(\l^i_\Maxi\right)^d .
\end{align*}
Together, these imply
 \begin{align}\label{eq:cof_ineq}
    &\mathrm{cof}(\bm{A}^i)=
        \det(\bm{A}^i)(\bm{A}^i)^{-1}\begin{cases}
        \preceq \frac{\Pi_{j=1}^d \l_j^i}{\min \l_j^i} \mathbb{I}_d\preceq \left(\l_\Maxi^i\right)^{d-1} \mathbb{I}_d,\\
        \succeq (\det(\mathrm{D}^2 u)/\l^i_\Maxi) \mathbb{I}_d\overset{\eqref{eq:boundedness}}\succeq \max\limits_\Om\{f_\Mini/\l^i_\Maxi,\, \l_\Mini^{d-1}\} \mathbb{I}_d.
    \end{cases}
\end{align} 
We define $\xi:=\max\limits_\Om\{f_\Mini/\l^i_\Maxi,\, \l_\Mini^{d-1}\}$.

\begin{proof}[\textbf{Proof of \Cref{lem:Classical}}]
Using the definitions of the L-scheme \eqref{eq:L-scheme} and the linearity of the Laplacian, we expand
\begin{equation}\label{eq:err_exp}
  \D e^{i+1} = \D (u^{i+1} -u) = \D(u^i + v^{i+1} - u) = \D e^i + \D v^{i+1} = \D e^i - \frac{\rho(u^i)}{\Lambda^i}.
\end{equation}
Since $\det (\mathrm{D}^2u) = f(\cdot,\del u),$ the residual $\r(u^i)$ is given by
\begin{align*}
  \rho(u^i) &= \det(\mathrm{D}^2u^i) - f(\cdot,\del u^i)= \det(\mathrm{D}^2u^i) - f(\cdot,\del u) + f(\cdot,\del u)  - f(\cdot,\del u^i)\\
    &= \det(\mathrm{D}^2u^i) - \det(\mathrm{D}^2u) - e_f^i . 
\end{align*}
Since both $\mathrm{D}^2u^i$ and $\mathrm{D}^2u$ are symmetric, we conclude using \eqref{eq:DetMeanValue} that there exists a $t \in (0,1)$, such that 
\begin{equation}\label{eq:residual_proof}
  \rho(u^i) = \text{cof}(t\mathrm{D}^2u^i+ (1-t)\mathrm{D}^2u)\bm{:} \mathrm{D}^2e^i  - e_f^i=\text{cof}(\bm{A}^i)\bm{:} \mathrm{D}^2e^i  - e_f^i .
\end{equation}
Substituting \eqref{eq:residual_proof} in \eqref{eq:err_exp} leads to the equation
\begin{equation}\label{eq:rewr_err}
  \D e^{i+1} -\frac{e_f^i}{\Lambda^i} = \D e^i - \frac{1}{\Lambda^i}(\text{cof}(\bm{A^i})\bm{:}\mathrm{D}^2e^i).
\end{equation} 
Using $\D e^i = \mathbb{I}_d\bm{:} \mathrm{D}^2 e^i$ and the linearity of the Frobenius product, we can write \eqref{eq:rewr_err} as 
\begin{equation}\label{eq:err_step2}
  \D e^{i + 1}  -\frac{e_f^i}{\Lambda^i}= \D e^i - \frac{1}{\Lambda^i} \text{cof}(\bm{A}^i) \bm{:} \mathrm{D}^2 e^i   = \left(\mathbb{I}_d - \frac{1}{\Lambda^i} \text{cof}(\bm{A}^i) \right) \bm{:} \mathrm{D}^2 e^i . 
\end{equation}
Since $\bm{A}^i$ is symmetric, using the spectral decomposition, we diagonalise $\bm{A}^i$ as 
\begin{equation}\label{eq:orthog_c}
  \text{cof}(\bm{A}^i) = \bm{Q}^i_A\,\bm{D}^i_A\,(\bm{Q}^i_A)^\mathrm{T},
\end{equation}
for some orthogonal matrix $\bm{Q}^i_A \in \R^{d\times d}$ and a diagonal matrix $\bm{D}^i_A \in \R^{d\times d}$, where the diagonal elements correspond to the eigenvalues of $\bm{A}^i$.  Let $\a_1^i \leq \ldots \leq \a_d^i$ denote the eigenvalues of $\text{cof}(\bm{A}^i)$. Then $(\bm{D}^i_A)_{kl}=\a_k^i \d_{kl}$, where $\d_{kl}$ is the Kronecker delta, and 
\begin{equation}\label{eq:aij_bounds}
    \xi\leq \a_j^i \leq \left(\l_\Maxi^i\right)^{d-1}
\end{equation}
from \eqref{eq:cof_ineq} with $\xi$ defined after \eqref{eq:cof_ineq}.
Since $\bm{Q}^i_A$ is orthogonal, we have $(\bm{Q}^i_A)^{-1} = (\bm{Q}^i_A)^{T}$. Substituting \eqref{eq:orthog_c} in \eqref{eq:err_step2} leads to 
\begin{equation*}
  \D e^{i+1}  -\frac{e_f^i}{\Lambda^i} =\left (\mathbb{I}_d -  \frac{1}{\Lambda^i}\bm{Q}^i_A \bm{D}^i_A (\bm{Q}^i_A)^\mathrm{T}\right) \bm{:} \mathrm{D}^2 e^i   = \left( \bm{Q}^i_A\left(\mathbb{I}_d -  \frac{1}{\Lambda^i}\bm{D}^i_A\right) (\bm{Q}^i_A)^\mathrm{T}\right) \bm{:} \mathrm{D}^{2}e^i .
\end{equation*}
Using the definition of the Frobenius product and the symmetry property of the trace, the equation above can be written after defining $\bm{B}^i:= (\bm{Q}^i_A)^{\mathrm{T}}(\mathrm{D}^2e^i)\bm{Q}^i_A $ as
\begin{align}
  \D e^{i+1}  -\frac{e_f^i}{\Lambda^i} &= \tr( \bm{Q}^i_A\left(\mathbb{I}_d - \frac{1}{\Lambda^i}\bm{D}^i_A\right) (\bm{Q}^i_A)^\mathrm{T} (\mathrm{D}^{2}e^i))\nonumber\\ &= \tr( \left(\mathbb{I}_d - \frac{1}{\Lambda^i}\bm{D}^i_A\right) (\bm{Q}^i_A)^\mathrm{T} (\mathrm{D}^{2}e^i)\bm{Q}^i_A)\nonumber\\
  &= \sum_{j,k=1}^d \left(\mathbb{I}_d-\frac{1}{\Lambda^i}\bm{D}_A^i\right)_{jk}(\bm{B}^i)_{jk}   = \sum_{j = 1}^d \left(1 - \frac{\a_j^i}{\Lambda^i}\right)(\bm{B}^i)_{jj}.\label{eq:err_step3}
\end{align}
The matrix $\bm{B}^i$ is positive semi-definite as $\mathrm{D}^2 e^i$ is positive semi-definite since
\begin{align*}
    \bm{y}^\mathrm{T} (\bm{Q}^i_A)^{\mathrm{T}}(\mathrm{D}^2e^i)\bm{Q}^i_A \bm{y}=  (\bm{Q}^i_A\bm{y})^{\mathrm{T}}(\mathrm{D}^2e^i)(\bm{Q}^i_A \bm{y})\geq 0, \quad \forall\, \bm{y}\in \R^d.
\end{align*} 
In particular, inserting $\bm{y}=\hat{\bm{e}}_j$ ($j^\mathrm{th}$ unit vector) we get $(\bm{B}^i)_{jj} \geq 0$ for $j = 1,2,\ldots,d$. 
Since $\Lambda^i \geq \|\l_\Maxi^i\|_{L^\infty(\Om)}^{d-1}$, we find the following bounds using \eqref{eq:aij_bounds},
\begin{equation}\label{eq:err_bounds}
  0\leq  1- \frac{\left(\l_{\Maxi}^i\right)^{d-1}}{\Lambda^i} \leq 1- \frac{\a_j^i}{\Lambda^i} \leq 1- \frac{\xi}{\Lambda^i}.
\end{equation}
This gives,
\begin{align}\label{eq:upper_lower}
  \D e^{i+1}  -\frac{e_f^i}{\Lambda^i} \begin{cases}
      &\leq \left(1- \dfrac{\xi}{\Lambda^i}\right) \sum_{j=1}^d (\bm{B}^i)_{jj}=\left(1- \dfrac{\xi}{\Lambda^i}\right)  \mathrm{tr}(\bm{B}^i)\\[.5em]
      &\geq \left(1- \dfrac{\left(\l_{\Maxi}^i\right)^{d-1}}{\Lambda^i}\right) \sum_{j=1}^d (\bm{B}^)i_{jj}=\left(1- \dfrac{\left(\l_{\Maxi}^i\right)^{d-1}}{\Lambda^i}\right)\mathrm{tr}(\bm{B}^i).
  \end{cases}
\end{align}
Using the definition of $\bm{B}^i$, the symmetric property of the trace and orthogonality of $\bm{Q}^i_A$, we find
\begin{equation*}
  \tr(\bm{B}^i) = \tr\left((\bm{Q}^i_A)^\mathrm{T}(\mathrm{D}^2e^i)\bm{Q}^i_A\right) = \tr (\mathrm{D}^2e^i) = \D e^i\geq 0.
\end{equation*}
Inserting in \eqref{eq:upper_lower} we prove \eqref{eq:to_proof_thm}.
\end{proof}

\begin{proof}[\textbf{Proof of \Cref{thm:convergence}}]
\textbf{(Step 1) Proving \eqref{eq:contraction_L2}.}
 Using the positivity of terms in \eqref{eq:to_proof_thm}, we use the following inequality
\begin{align*}
    \left\Vert \D e^{i+1} -\frac{e_f^i}{\Lambda^i}  \right\Vert_{L^2(\Om)} &\leq \left\Vert \left(1- \dfrac{\xi}{\Lambda^i}\right)\D e^i \right\Vert_{L^2(\Om)} .
\end{align*}
We can rewrite this equation using \eqref{eq:lipschitz_of_f}, \Cref{lem:lemma_for_proof} and the reverse triangle inequality as
\begin{align*}
     \Vert \D e^{i+1} \Vert_{L^2(\Om)} &\leq  \left\Vert \left(1- \dfrac{\xi}{\Lambda^i}\right)\D e^i \right\Vert_{L^2(\Om)}+ \left\Vert \frac{e_f^i}{\Lambda^i} \right\Vert_{L^2(\Om)}\\
     &\leq  \left\Vert \left(1- \dfrac{\xi}{\Lambda^i}\right)\D e^i \right\Vert_{L^2(\Om)} \overset{\ref{ass:Afa}}+ \frac{\mu_f}{\Lambda^i}\Vert \del e^i\Vert_{L^2(\Om)}\\
    &\leq \left(1- \dfrac{\xi}{\Lambda^i}\right)\left\Vert \D e^i \right\Vert_{L^2(\Om)} \overset{\Cref{lem:lemma_for_proof}}+  \frac{C_E\mu_f}{\Lambda^i}  \Vert \D e^i\Vert_{L^2(\Om)} \\
    &=  \left(1 - \dfrac{\xi}{\Lambda^i} + \frac{C_E \mu_f}{\Lambda^i}\right) \Vert \D e^i\Vert_{L^2(\Om)}, 
\end{align*}
where $C_E \geq 1$ is a constant independent of $u^i$.
The contraction rate becomes
\begin{align*}
   &q_i:= 1 - \dfrac{\xi}{\Lambda^i} + \frac{ C_E\mu_f}{\Lambda^i}< 1, 
\end{align*}
provided $\mu_f< \xi/C_E=:\bar{\xi}_{f,u}$. Observe that $\bar{\xi}_{f,u}$ depends on $f$ and $u$, and is bounded from below for $\mu_f\searrow 0$ since both $f$ is bounded from below, and $\mathrm{D}^2 u$ has uniformly strictly positive and bounded eigenvalues in this limit. Thus, there exists $\mu_f^*>0$, dependent on $\Om$ and $f(\bm{\cdot},\bm{0})$, such that for $\mu_f<\mu_f^*$, we have $q_i<1$. 


\textbf{(Step 2) Proving (i).}  Recall that $\mu_f=0$ for this part.
Then, as a result of \eqref{eq:to_proof_thm}, we have $\D e^{i+1} < \D e^{i}$ for all $ i \geq 0$. By an inductive argument, we find the decreasing sequence   $0 \leq \D e^{i+1} < \D e^i < \ldots < \D e^0$.
Using $u^i = u+ e^i$ for all $i \geq 0$, we find another decreasing sequence 
\begin{equation*}
  0 < \D u \leq \D u^{i+1} < \D u^i < \ldots < \D u^0.
\end{equation*}
Using $\D u^i = \tr(\mathrm{D}^2u^i)=\sum_{j=1}^d \l^i_j$, we express
\begin{equation}\label{eq:lamb_id}
\sum_{j=1}^d \l^i_j\leq \D u^0 \;\;\text{ where }\;\; \l^i_j\geq \l_\Mini \;\text{ for all } j=1,\dots,d,
\end{equation}
due to \eqref{eq:D2u_bounds}. Then we have
\begin{equation}
    \l_\Maxi^i\leq \del u^0 -(d-1)\l_\Mini \text{ almost everywhere in } \Om.
\end{equation}
Thus, setting $\Lambda^i=\|\D u^0 -(d-1)\l_\Mini\|^{d-1}_{L^\infty(\Om)}$, we guarantee that $\Lambda^i\geq \|\l_\Maxi^i\|_{L^\infty(\Om)}^{d-1}$ for all $i\in\N$.

\textbf{(Step 3) Proving (ii).} 
By the assumption in (ii), there exists a $\bar{\Lambda} > 0$ such that $\Lambda^i < \bar{\Lambda} $ for all $i \in \N$. This implies due to \eqref{eq:lumbed_constant} that $\|\l^i_\Maxi\|_{L^\infty(\Om)}\leq \bar{\Lambda}^{\frac{1}{d-1}}$.
Then, by \eqref{eq:to_proof_thm} we have almost everywhere in $\Om$,
\begin{equation*}
  0\leq \D e^{i+1} \leq \left(1 - \frac{f_{\Mini}}{\l^i_\Maxi\Lambda^i}\right) \D e^i \leq \left(1 - \frac{f_{\Mini}}{\bar{\Lambda}^{\frac{d}{d-1}}}\right) \D e^i =: q\D e^i.
\end{equation*}
This means $\|\D e^{i+1}\|_{L^p(\Om)}\leq q \|\D e^{i}\|_{L^p(\Om)}$ for any $p\in [1,\infty]$, i.e., this $L^p$-error measure is contractive. The linear convergence with rate $q<1$ in $H^2(\Om)$ is special case for $p=2$, since $\|\D (\cdot)\|_{L^2(\Om)}$ is an equivalent norm on $H^2(\Om)\cap H^1_0(\Om)$ by \Cref{lem:lemma_for_proof}. Finally, since $\|\D e^{i}\|_{L^\infty(\Om)}\to 0$, using \cite[Theorem 8.34]{gilbarg1977elliptic} we get that $u^i\to u$ in $C^{1,1}(\bar{\Om})$.
\end{proof}

\begin{remark}[Choice of $\Lambda^i$]
    During the proof, we showed pointwise convergence which holds if $\Lambda^i \geq \|\l_{\Maxi}^i\|^{d-1}_{L^\infty(\Om)}$. While  \eqref{eq:lumbed_constant} is sufficient for this, it is not necessary.
Moreover, we could construct $\Lambda^i$ to be dependent on $\bm{x} \in \Om$. Considering convergence, this analysis is valid as long as $\Lambda^i(\bm{x}) \geq \l_{\Maxi}^i (\bm{x})$ for all $\bm{x} \in \Om$. This is the direction pursued in \cite{MITRA20191722,javed2025robust} and it leads to faster convergence. However, keeping $\Lambda^i$ constant in $\Om$ enables us to implement other acceleration methods, see \Cref{sec:two-strategies}. Nevertheless, $\Lambda^i$ can be chosen adaptively every iteration based on $\l_\Maxi^i$ values to expedite the convergence. This is implemented in \Cref{sec:numRes} inspired by the adaptive linearisation method in \cite{stokke2023adaptive}.
\end{remark}

\subsection{Convergence of generalised solutions}\label{sec:generalised}

\begin{theorem}[Linear convergence in $L^\infty$ to viscosity solutions] 
Let \ref{ass:Afa}--\ref{ass:Aga} hold with $\mu_f=0$, and $\Om$ be a convex domain with a $C^2$-boundary. Let $u\in C(\bar{\Om})$ be a generalised solution of \eqref{eq:main_problem} in terms of \Cref{def:gen_sol} (see \Cref{prop:gen_sol}) such that there exists $\l_{\Mini}>0$ for which 
    \[
    u(\bm{y})-u(\bm{x})\geq \bm{p}\cdot(\bm{y}-\bm{x}) + \frac{1}{2}\l_\Mini |\bm{y}-\bm{x}|^2,
    \]
    for all $\bm{x},\,\bm{y}\in \bar{\Om}$ and $\bm{p}\in \p u(\bm{x})$.  
Let $\{u^{i}\}_{i\in \N}\in C^{1,1}(\bar{\Om})$ denote the L-scheme iterates generated by solving \eqref{eq:L-scheme}. 
 We fix an iteration index $i\in \N$, and let $u^i \in C^{1,1}(\bar{\Om})$ satisfy $u^i=\gamma$ on $\p\Om$, and for a constant $\l_\Maxi^i>0$
     \[
    u^i(\bm{y})-u^i(\bm{x})\leq \nabla u^i(\bm{x})\cdot(\bm{y}-\bm{x}) + \frac{1}{2}\l_\Maxi^i |\bm{y}-\bm{x}|^2,
    \]
    for all $\bm{x},\,\bm{y}\in \bar{\Om}$. 
Choose a lumped constant $\Lambda^i>0$ satisfying
  \begin{equation}\label{eq:lumbed_con}
    \Lambda^i \geq (\l_{\Maxi}^i)^{d-1}.
  \end{equation}
Assume that the error $e^i=u^i-u$ is convex in $\Om$. Then, 
    \begin{align}
        \left( 1-\frac{\l_\Mini^{d-1}}{\Lambda^i}\right) e^{i+1} \leq e^i  \leq  0, \text{ a.e. in } \Om. 
    \end{align}   
Moreover, if $e^i$ are convex for all $i\in \N$, and $\Lambda^i$ satisfying \eqref{eq:lumbed_con} are bounded from above by some $\Tilde{\Lambda} > 0$ then $u^i$ converges to $u$ linearly in $L^\infty(\Om)$ with contraction rate $\bar{q}=1-(\l_\Mini^{d-1}\slash \Tilde{\Lambda})$.
    \label{thm:Linfty}
\end{theorem}
\noindent
The proof of the theorem above is complicated by the fact that $\mathrm{D}^2 u$ might not be well-defined for generalised solutions. Thus, the guaranteed contraction rate is also different from \Cref{thm:convergence}. We prove the theorem by passing to the limit of regularised solutions.

\begin{proof} \textbf{(Step 1) The domain $\Om_\e$ and its properties.}
    Since the domain $\Om$ is convex with $C^2$ boundaries, there exists a convex function $G_\Om\in C^{1,1}(\bar{\Om}) $ such that $G_\Om|_{\p\Om}=0$, i.e., $\p\Om$ is the 0 level-set of $G_\Om$. Such a function could simply be the solution of a Monge-Ampère equation with homogeneous  Dirichlet condition.  Observe that since $G_\Om$ is convex, $G_\Om<0$ in $\Om$.   For $\e>0$, we define the set
    \begin{align*}
        \Om_\e:= \{\bm{x}\in \Om \;\;|\;\; G_\Om(\bm{x})<-\e\}.
    \end{align*}
    Due to the convexity of $G_\Om$, the set $\Om_\e$ is convex. Since multiplying $G_\Om$ by a positive constant preserves its 0 level-set and convexity, we can assume that $G_\Om$ has a Lipschitz constant of 1. This gives,
    \begin{align*}
        \mathrm{dist}(\p\Om,\Om_\e)>\e \text{ for } \e=|G_\Om(\bm{x})-G_\Om(\bm{y})|\leq |\bm{x}-\bm{y}| \text{ for all } \bm{x}\in \p\Om \text{ and } \bm{y}\in \p\Om_\e.
    \end{align*}

    \textbf{(Step 2) The functions $u_\e,\, u^i_\e$ and their properties.} Let $\eta_\e$ be the standard mollifying function defined, e.g.,  in Appendix C of \cite{evans2022partial}. We introduce approximating functions $u_\e,\, u^{i}_\e:\Om_\e\to \R$ defined as
\begin{equation}
    u_\e:= u\ast \eta_\e= \int_{\R^d}  u(\bm{x} + \bm{h}) \,\eta_\varepsilon (\bm{h}) \,\mathrm{d}\bm{h},  u_\e^i:= u^i\ast \eta_\e= \int_{\R^d}  u^i(\bm{x} + \bm{h}) \,\eta_\varepsilon (\bm{h})\, \mathrm{d}\bm{h}.
\end{equation}
 Due to $\mathrm{dist}(\p\Om,\Om_\e)>\e$, these functions are well-defined.   First, we show that $u_\varepsilon$ is convex in $\Om_\varepsilon$.
For all $\bm{x},\bm{y}\in\Om_\varepsilon$ we have
\begin{align*}
    u_\varepsilon (t \bm{x} + (1-t)\bm{y}) &= \int_{\R^d} u((t \bm{x} + (1-t)\bm{y}) +\bm{h})\,\eta_\varepsilon (\bm{h})\, \mathrm{d}\bm{h} \\ 
    &=  \int_{\R^d} u(t (\bm{x}+\bm{h}) + (1-t)(\bm{y}+\bm{h}))\,\eta_\varepsilon (\bm{h})\, \mathrm{d}\bm{h} \\
    &\leq t \int_{\R^d} u(\bm{x}+\bm{h}) \,\eta_\varepsilon (\bm{h}) \,\mathrm{d}\bm{h} + (1-t) \int_{\R^d} u(\bm{y}+\bm{h}) \,\eta_\varepsilon (\bm{h})\, \mathrm{d}\bm{h} \\
    &= t u_\varepsilon (\bm{x}) + (1- t) u_\varepsilon (\bm{y}).
\end{align*}
Similarly $u_\varepsilon^i$ is also convex.
Moreover, since $(u^i -u)$ is convex, so is ($u^i-u)_\varepsilon = u_\varepsilon^i - u_\varepsilon$.
Thus, $\mathrm{D}^2 u^i_\e\succeq \mathrm{D}^2 u_\e$. Now, since $u$ is convex in $\Omega$, it is Lipschitz \cite{nguyen2021lipschitz}, which implies by Rademacher's theorem that it is almost everywhere differentiable. Thus, for almost all $\x\in \Om_\e$, and $|\d \bm{x}|<\e$
\begin{equation*}
    u(\bm{x} + \d \bm{x}) \geq  u(\bm{x}) + \del u(\bm{x})\cdot \d \bm{x} + \frac{1}{2} \l_\Mini \vert \d \bm{x} \vert ^2 .
\end{equation*}
Multiplying with $\eta_\varepsilon$ and integrating 
\begin{align*}
    u_\varepsilon(\bm{x} + \d \bm{x}) &= \int_{\R^d} u(\bm{x} +\d\bm{x}+\bm{h})\,\eta_\varepsilon (\bm{h}) \, \mathrm{d} h \\
    &\geq \int_{\R^d} u(\bm{x}+\bm{h}) \,\eta_\varepsilon (\bm{h})\, \mathrm{d}\bm{h} + \d\bm{x} \bm{\cdot} \int_{\R^d} \nabla u(\bm{x}+\bm{h}) \eta_\varepsilon (\bm{h}) \mathrm{d}\bm{h} + \frac{\l_\Mini}{2} \vert \d \bm{x} \vert^2 \int_{\R^d} \eta_\varepsilon (\bm{h}) \mathrm{d}\bm{h} \\
    &= u_\varepsilon(\bm{x}) + \d \bm{x} \bm{\cdot} \del u_\varepsilon (\bm{x}) + \frac{\l_\Mini}{2} \vert \d \bm{x} \vert^2 .
\end{align*}
Since $u_\varepsilon \in C^\infty(\Om)$ it implies $\mathrm{D}^2u_\varepsilon \succeq \l_\Mini \mathbb{I}_d$ by using the Taylor-series and passing $|\d\bm{x}|\to 0$.

Rademacher's theorem also implies that $\mathrm{D}^2u^i$ exists almost everywhere since $u^i\in C^{1,1}(\bar{\Om})$. Thus, similar to before, using
\begin{equation*}
    u^{i}(\bm{x} +\d \bm{x}) \leq u^i(\bm{x}) + \d \bm{x} \cdot \del u^i (\bm{x}) + \frac{\l_\Maxi^i}{2} \vert \d \bm{x} \vert^2,
\end{equation*}
we have $  \mathrm{D}^2 u_\varepsilon^i \preceq \l_\Maxi^i \mathbb{I}_d$.
Combining these estimates, we have the ordering 
\begin{equation*}
    \l_\Mini \mathbb{I}_d \preceq \mathrm{D}^2u_\varepsilon \preceq \mathrm{D}^2u_\varepsilon^i \preceq \l_\Maxi^i \mathbb{I}_d \preceq (\Lambda^i)^{\frac{1}{d-1}} \mathbb{I}_d.
\end{equation*}

\textbf{(Step 3) The updated $u^{i+1}_\e$ and its properties.}
Now, define 
\begin{equation}
f_\varepsilon := \det(\mathrm{D}^2u_\varepsilon)= \mathcal{M} u_\e, \text{ with $\mathcal{M}$ as the Monge-measure in \eqref{eq:monge-measure}.}
\end{equation}
 
With these definitions, update $u^{i+1}_\e\in H^2(\Om)\cap H^1_0(\Om)$ by solving
\begin{equation}\label{eq:uepsi_update}
\left\{
\begin{aligned}
\D (u_\varepsilon^{i+1} -u_\varepsilon^i) &= f_\varepsilon - \det(\mathrm{D}^2 u_\varepsilon^i)  &\text{in } \Om_\varepsilon, \\
u_\varepsilon^{i+1} &= u_\varepsilon + \left( 1- \frac{\l_\Mini^{d-1}}{\Lambda^i}\right)(u_\varepsilon^i - u_\varepsilon)  &\text{on } \partial \Om_\varepsilon .
\end{aligned}
\right.
\end{equation}
 Defining $\bm{A}^i_\e:= t \mathrm{D}^2 u^i_\e +(1-t)\mathrm{D}^2 u^i$ for some $t\in [0,1]$ and observing that for all $\e>0$,
\[
\l_\Mini^{d-1} \mathbb{I}_d \preceq \mathrm{cof}(\bm{A}^i_\e)\preceq (\l_\Maxi^i)^{d-1} \mathbb{I}_d \]
similar to \eqref{eq:cof_ineq},  we repeat the analysis of  \Cref{lem:Classical} to get
    \begin{equation}\label{eq:contraction_eps}
        0\leq \D (u_\varepsilon^{i+1} -u_\varepsilon) \leq \left(  1- \frac{\l_\Mini^{d-1}}{\Lambda^i}\right) \D (u_\varepsilon^i - u_\varepsilon)  \quad \text{in } \Om_\varepsilon.
    \end{equation}
This gives by defining 
    \begin{equation*}
     \Upsilon_\e^i:= (u_\varepsilon^{i+1} -u_\varepsilon) -  \left(  1- \frac{\l_\Mini^{d-1}}{\Lambda^i}\right) (u_\varepsilon^i - u_\varepsilon) \text{ that }  \begin{cases} \D \Upsilon_\e^i \leq 0 &\text{ in } \Om_\e,\\
     \Upsilon_\e^i=0  &\text{ on } \p\Om_\e.
     \end{cases}
    \end{equation*}   
Consequently, by the Maximum principle \cite[Chapter 6.4]{evans2022partial},  $\Upsilon_\e^i\geq 0$ a.e. in $\Om_\e$. 

Moreover, from \eqref{eq:contraction_eps} we get that 
\begin{align*}
\D (u_\varepsilon^{i+1} -u_\varepsilon)\geq 0 \text{ in } \Om_\e \text{ and }    u_\varepsilon^{i+1} - u_\varepsilon =\left(  1- \frac{\l_\Mini^{d-1}}{\Lambda^i}\right) (u_\varepsilon^i - u_\varepsilon)\leq 0 \quad \text{in } \p\Omega_\varepsilon .
\end{align*}
In the last inequality we have used that $u^i_\e\leq u_\e$ in $\Om_\e$ since $u^i\leq u$ in $\Om$. This is because $e^i=u^i-u$ is convex, and vanishing on $\p\Om$, thus, making $e^i\leq 0$ in $\Om$.  Again from the Maximum principle we get that $u_\varepsilon^{i+1} -u_\varepsilon\leq 0$ a.e. in $\Om_\e$. Combining, we get that 
\begin{equation}\label{eq:Linfty_contrct_eps}
   0 \geq  u_\varepsilon^{i+1} - u_\varepsilon \geq \left(  1- \frac{\l_\Mini^{d-1}}{\Lambda^i}\right) (u_\varepsilon^i - u_\varepsilon)\quad \text{in } \Omega_\varepsilon .
\end{equation}

\textbf{(Step 4) Passing the limit $\e\searrow 0$.}
Multiplying the first equation in \eqref{eq:uepsi_update} with $w\in C^\infty(\Om)$ such that support of $w$ is compactly embedded in $\Om$, we get for $\e$ small enough such that $\mathrm{supp} (w)\Subset \Om_\e$ the weak form
\begin{align}
    \int_{\Om} \del (u^{i+1}_\e -u^i_\e)\bm{\cdot} \del w \,\mathrm{d}\bm{x} =\int_{\Om} w\det (\mathrm{D}^2 u^i_\e) \,\mathrm{d}\bm{x} - \int_{\Om} w f_\e \,\\mathrm{d}\bm{x}.
\end{align}
Observe that from Guiterezz \cite[Lemma 1.2.3]{gutierrez2001monge}, that $u_\e\to u$ uniformly,  $f^\varepsilon =\mathcal{M}u_\e \rightharpoonup f$ weakly, i.e., $\int_{\Om} w f_\e \mathrm{d}\bm{x} \to \int_{\Om} \mathrm{d}\bm{x} w f$ as $\e\searrow 0$.
Similarly, since $u^i_\e\to u^i$ uniformly in $C^{1,1}(\Om')$ for all $\Om'\Subset \Om$, we have by passing the limit $\e\searrow 0$ that
\begin{equation}
     u^{i+1}_\varepsilon \to u^{i+1}.
\end{equation}
Sending $\varepsilon \searrow 0$ in \eqref{eq:Linfty_contrct_eps}, we finally get 
\begin{equation*}
    0\geq u^{i+1} - u \geq \left(1 - \frac{\l_\Mini^{d-1}}{\Lambda^i}\right) (u^i - u).
\end{equation*} 
\end{proof}

\section{Two fast solution strategies}\label{sec:two-strategies}
For simplicity in introducing the solvers, we consider $\Om = (0,1)^2$ as the domain, implying that $d=2$. This will also be used in the numerical experiments in \Cref{sec:numRes}. In the following, a finite difference discretisation method will be examined. The domain $\Om = (0,1)^2$ is discretised as follows.
For $N \in \N$, we introduce a uniform grid size $\D x := \frac{1}{N + 1}>0$.
Then, the $(N+2)\times (N+2)$ computational grid is defined as 
\begin{equation*}
  \bm{x}_{j,k} := (x_j, y_k), \quad x_j := j \D x, \quad y_k := k\D x, \quad j,k = 0,1,\ldots,N+1.
\end{equation*}
The discretised solution on the interior grid is denoted by $u_{j,k} \approx  u(\bm{x}_{j,k})$ for $j,k = 1,2, \ldots, N$. 
The approximations are stored in lexicographical order as 
\begin{equation}
  \underline{u} := (u_{1,1},\ldots,u_{N,1},\ldots,u_{1,N},\ldots,u_{N,N})^\mathrm{T} .
\end{equation}
All underlined vectors in this context have the dimension $\R^{N^2}$. In a similar manner, $f$ is discretised on the interior grid.

Using central differences, the second-order partial derivatives of $u$ are approximated at $\bm{x}_{jk}$ for $j,k = 2,\ldots , N-1$ as
\begin{subequations}
\begin{align}\label{eq:approx_second}
  u_{xx}(\bm{x}_{j,k}) &\approx \frac{1}{\D x^2}(u_{j-1,k}-2u_{j,k}+u_{j+1,k}), \\
  u_{xy}(\bm{x}_{j,k}) &\approx \frac{1}{ 4\D x^2}(u_{j-1,k-1}-u_{j-1,k+1} - u_{j+1,k-1} + u_{j+1,k+1}), \\
  u_{yy}(\bm{x}_{j,k}) &\approx \frac{1}{\D x^2}(u_{j,k-1}-2u_{j,k}+u_{j,k+1}) .
\end{align}
\end{subequations}
If $ j\in \{1,N\}$ or $k \in \{1,N\}$, the second order partial derivatives can be approximated similarly using the Dirichlet boundary conditions. 

Let $\underline{u}^i$ denote the current iterate for $i\in \N$.
We introduce the vector of the discretised Hessian determinant and residual as
\begin{equation*}
  \underline{f}^i := \underline{u}_{xx}^{i} \odot \underline{u}_{yy}^{i} - \underline{u} _{xy}^{i} \odot  \underline{u} _{xy}^{i}, \;\;\text{ and } \;\; \underline{\rho}^i:= \underline{f}^i-\underline{f},
\end{equation*}
where $\odot$ denotes element-wise multiplication, called Hadamard product.
Similarly, we introduce the vector of the discretised Hessian traces, and the maximum eigenvalues of the Hessian matrices
\begin{equation*}
  \underline{\tau}^{i}  := \underline{u} _{xx}^{i} +\underline{u}_{yy}^{i} \;\;\text{ and } \;\;  \underline{\l}_{\Maxi}^{i} := \frac{1}{2}\left(\underline{\tau}^{i} + \sqrt{\underline{\tau}^{i} \odot \underline{\tau}^{i} -4 \underline{f}^i}\right), 
\end{equation*}
where the root is taken element-wise. Conforming with the convergence criteria in \Cref{thm:convergence,thm:Linfty} we need the lumped constant to satisfy $\Lambda^i \geq \max \underline{\l}_{\Maxi}^{i}$. 
As we only measure $\underline{\l}_{\Maxi}^{i}$ on the computational grid, we introduce a safety parameter $\eta \geq 1$ and select
\begin{equation*}
  \Lambda^i = \eta \Vert \underline{\l}_{\Maxi}^{i}\Vert_\infty .
\end{equation*}

\subsection{Green's Function}
One approach to solve the Poisson equation update in \eqref{eq:L-scheme} is by using Green's representation formula. For a linear homogeneous Dirichlet problem,
\begin{equation}\label{eq:arb_green}
    \left\{
  \begin{aligned}
    \mathcal{L}[u] &= f, &\text{ in } \Om,\\
    u &= 0, &\text{ on } \partial \Om,
  \end{aligned}
  \right.
\end{equation}
where $\mathcal{L}[u]$ is the associated differential operator of a linear second-order elliptic PDE, the solution can be expressed in integral form as
\begin{equation}
    u(\bm{x_0}) = \int_\Om G(\bm{x};\bm{x_0})\, f(\bm{x})\,\mathrm{d} \bm{x}.
\end{equation}
Here, $G(\bm{x};\bm{x_0})$ is the associated Green's function of \eqref{eq:arb_green} that satisfies the homogeneous Dirichlet problem 
\begin{equation*}
    \mathcal{L}[G](\bm{x};\bm{x}_0)=\d (\bm{x}-\bm{x}_0) \qquad\text{ in } \Omega,
\end{equation*}
with $\d$ being the Dirac distribution, see Chapter 2.2 of \cite{evans2022partial}.
Then, for the Poisson updates of \eqref{eq:poisson_iter} the explicit representation becomes
\begin{equation}\label{eq:green_int}
  v^{i+1}(\bm{x}_0) = -\frac{1}{\Lambda^i} \int_\Om G_{\rm P}(\bm{x};\bm{x}_0)\, \rho(u^{i}(\bm{x}))\, \mathrm{d}\bm{x},
\end{equation}
where $G_{\rm P}(\bm{x};\bm{x}_0)$ is the Green's function of the Poisson equation with homogeneous Dirichlet boundary condition which is explicitly computable \cite{zauderer2011partial}. More precisely,
\begin{equation}
  G_{\rm P}(\bm{x},(x_0,y_0)) = 4 \sum_{m,n=1}^\infty \frac{\sin(m\pi x)\sin(n\pi y)\sin(m\pi x_0)\sin(n\pi y_0)}{\pi^2 (m^2+n^2)}.
\end{equation}
Using a truncated series, we can approximate this Green's function $G_{\rm P,M}$. 
We approximate the integral in \eqref{eq:green_int} at discrete points $\bm{x}_{j,k}$ using the composite trapezoidal rule, which gives the value of the update $v^{i+1}_{j,k}\approx v^{i+1}(\bm{x}_{j,k})$ as
\begin{equation}\label{eq:green_trapez}
  v_{j,k}^{i+1} \approx-\frac{\D x^2}{\Lambda^i} \sum_{m,n =1}^M G_{\rm P,M}(\bm{x}_{m,n};\bm{x}_{j,k})\,\rho(u^i(\bm{x}_{m,n})).
\end{equation}
Then we discretise the Green's function as a matrix $\underline{\bm{G}}_{\rm P,M} \in \R^{N^2 \times N^2}$ defined as 
\begin{equation}\label{eq:GP_matrix}
  \underline{\bm{G}}_{\rm P}:= \D x^2 \begin{pmatrix}
    G_{\rm P,M}(\bm{x}_{1,1};\bm{x}_{1,1}) & \cdots & G_{\rm P,M}(\bm{x}_{N,N};\bm{x}_{1,1}) \\
    \vdots                       & \ddots & \vdots \\
    G_{\rm P,M}(\bm{x}_{1,1};\bm{x}_{N,N}) & \cdots & G_{\rm P,M}(\bm{x}_{N,N};\bm{x}_{N,N})
  \end{pmatrix}.
\end{equation}
Using $\underline{\bm{G}}_{\rm P}$ and \eqref{eq:poisson_iter}, we rewrite \eqref{eq:green_trapez} in the matrix-vector multiplication form
\begin{equation*}
  \underline{v}^{i+1} := -\frac{1}{\Lambda^i} \underline{\bm{G}}_{\rm P} \underline{\rho}^i.
\end{equation*}
The stopping criterion of the fixed-point method is based on the vanishing values of the Poisson updates and can be defined as 
\begin{equation*}
  \Vert \underline{v}^{i+1}\Vert_2 \leq \delta_{\text{tol}},
\end{equation*}
for some tolerance $\delta_{\text{tol}} >0$. Furthermore, a maximum number of iteration steps $i_{\text{max}} \in \N$ and a threshold for the maximal eigenvalue  $\l_\text{thresh}$ are introduced. An overview of this fixed-point method with Green's function evaluations is provided in \Cref{algorithm1}.


\begin{algorithm}[H]
\caption{Fixed-point method with Green's function}
\label{algorithm1}
\textbf{Input:} $\underline{u}^0$, $\underline{f}$, $i_{\text{max}}$, $\delta_{\text{tol}}$, $\eta$, $\lambda_{\text{thresh}}$; \textbf{Precompute:} $\underline{\bm{G}}_{\rm P}$

\begin{algorithmic}[1]
\For{$i = 0$ to $i_{\text{max}}$}
    \State Compute $\;\;\underline{\tau}^{i} \leftarrow \underline{u}_{xx}^{i} +\underline{u}_{yy}^{i}, \quad$ $\underline{f}^{i} \leftarrow \underline{u}_{xx}^{i} \odot \underline{u}_{yy}^{i}  - \underline{u}_{xy}^{i} \odot \underline{u}_{xy}^{i},  \quad $ $\underline{\rho}^i \leftarrow  \underline{f}^i-\underline{f}$\vspace{.5em}
    \State Compute $\;\;\underline{\l}_{\Maxi}^i \leftarrow \frac{1}{2}(\underline{\tau}^i + \sqrt{\underline{\tau}^i \odot \underline{\tau}^i - 4 \underline{f}^i}),\;\;$ select $\;\; \Lambda^{i} \leftarrow \min(\eta \Vert \underline{\l}_{\Maxi}^{i} \Vert_\infty, \l_{\rm thresh})$\vspace{0.5em}
    \State Compute $ \underline{v}^{i+1}  \leftarrow -\frac{1}{\Lambda^{i}} \underline{\bm{G}}_{\rm P}\underline{\rho}^i,\;\;$ update $\;\; \underline{u}^{i+1} \leftarrow \underline{u}^{i}  + \underline{v}^{i+1} $ \vspace{0.5em}
    \If{$\left\| \underline{v}^{i+1} \right\|_2 \leq \delta_{\text{tol}}$}
        \State break
    \EndIf
\EndFor
\end{algorithmic}
\end{algorithm}

\subsection{Finite difference discretisation for the Laplacian}
An alternative approach to solving Poisson updates involves central finite differences to approximate discretized Poisson updates $v_{j,k}^{i+1} \approx v^{i+1}(\bm{x}_{j,k})$ for $j,k = 1, \ldots, N$. 
The discretisation and approximation are done similarly to the discretisation of the derivatives of the Hessian matrix. As a result, the finite difference approximation for the update becomes
\begin{equation}\label{eq:discr_poisson}
\frac{1}{\D x^2} (v_{j+1,k}^{i+1} + v_{j,k+1}^{i+1} + v_{j-1,k}^{i+1} + v_{j,k-1}^{i+1} - 4v_{j,k}^{i+1})   = -\frac{1}{\Lambda^i} \r(u^i(\bm{x}_{j,k})).
\end{equation}
This can be written in matrix-vector multiplication format as $\underline{\bm{A}}\underline{v}^{i+1} = -\frac{1}{\Lambda^i}\underline{\r}^i$, where the matrix $\underline{\bm{A}}$ corresponds to the discretisation of the Laplacian operator.
In this approach, the computation of $v^{i+1}$ in \Cref{algorithm1}, is modified: rather than evaluating the expression $-\frac{1}{\Lambda^i} \underline{\bm{G}}_{\rm P}\underline{\r}^i $, we solve the linear systems $\underline{\bm{A}}\underline{v}^{i+1} =-\frac{1}{\Lambda^i}\underline{\r}^i$ which is summarized in \Cref{algorithm2}. The matrix $\underline{\bm{A}}$ is sparse and stays the same at every iteration, and thus can be assembled once. The assembly of the matrix $\underline{\bm{A}}$ is significantly faster because it does not depend on the computation of the Green's functions, and it is a sparse matrix unlike $\underline{\bm{G}}_{\rm P}$. 
Moreover, because $\underline{\bm{A}}$ is sparse and symmetric positive definite, the system $\underline{\bm{A}}\underline{v}^{i+1} =-\frac{1}{\Lambda^i}\underline{\r}^i$ can be solved efficiently using the preconditioned conjugate gradient method.
The preconditioned system of equations is denoted by 
\begin{equation*}
    \underline{\bm{P}}^{-1} \underline{\bm{A}} \underline{v}^{i+1} = -\frac{1}{\Lambda^i}\underline{\bm{P}}^{-1} \underline{\r}^i,
\end{equation*}
where $\underline{\bm{P}}$ is a symmetric positive definite matrix \cite{demmel1997applied} which can be computed once, and applied at every iteration to speed up the solution process significantly. 

In \Cref{sec:numRes}, we will compare different types of preconditioners.
The first approach, employs an incomplete LU preconditioner (PCG: LU), as described in e.g. \cite{KAASSCHIETER1988265}.
An alternative preconditioning strategy involves using an algebraic multigrid (AMG) solver. Preconditioning with AMG will be referred to as PCG: AMG. Specifically, in the \texttt{pyamg} framework, we configure a V-cycle with a smoothed aggregation solver, as described by \cite{vanek1996algebraic,van2001convergence}. For comparison, we also solve the nonlinear system of equations using the AMG solver, applying the same settings as those used for preconditioning.
\begin{algorithm}[H]
  \caption{Fixed-point method with finite difference method \& preconditioning.}\label{algorithm2}
  \textbf{Input:} $\underline{u}^0,\underline{f},i_{\text{max}}, \delta_{\text{tol}},\eta, \l_{\rm thresh}$; \textbf{Assemble:} $\underline{\bm{A}}$; \\ \textbf{Precompute:} $\underline{\bm{P}}$
\begin{algorithmic}[1]
  \For{$i = 1 : i_{\text{max}}$}\vspace{.5em}
 \State Compute $\;\;\underline{\tau}^{i} \leftarrow \underline{u}_{xx}^{i} +\underline{u}_{yy}^{i}, \quad$ $\underline{f}^{i} \leftarrow \underline{u}_{xx}^{i} \odot \underline{u}_{yy}^{i}  - \underline{u}_{xy}^{i} \odot \underline{u}_{xy}^{i},  \quad $ $\underline{\rho}^i \leftarrow  \underline{f}^i-\underline{f}$\vspace{.5em}
    \State Compute $\;\;\underline{\l}_{\Maxi}^i \leftarrow \frac{1}{2}(\underline{\tau}^i + \sqrt{\underline{\tau}^i \odot \underline{\tau}^i - 4 \underline{f}^i}),\;\;$ select $\;\; \Lambda^{i} \leftarrow \min(\eta \Vert \underline{\l}_{\Maxi}^{i} \Vert_\infty, \l_{\rm thresh})$\vspace{0.5em}
    \State $ \underline{v}^{i+1} \leftarrow \texttt{solve }  \underline{\bm{P}}^{-1} \underline{\bm{A}} \underline{v} = -\frac{1}{\Lambda^i}\underline{\bm{P}}^{-1} \underline{\rho}^i,\;\;$ \text{ update } $\;\; \underline{u}^{i+1}\leftarrow \underline{u}^i + \underline{v}^{i+1}$\vspace{.5em}
    \If{$\Vert \underline{v}^{i+1} \Vert_2 \leq \delta_{\text{tol}}$ }
            \State break
    \EndIf
  \EndFor
\end{algorithmic}
\end{algorithm}

For Newton iteration, a finite difference discretisation of the Jacobian is computed similar to the previous case. However, the matrix changes every iteration and might become ill-conditioned or even singular.

\section{Numerical results}\label{sec:numRes}
To evaluate the performance of the iterative schemes in \Cref{sec:iterarive} and the solution strategies in \Cref{sec:two-strategies}, we investigate several test cases.

\subsection{Test cases and validation}
\subsubsection{Gaussian solution ($\mu_f\not=0$)} 
To demonstrate convergence of the solver to the correct solution, we selected a two-dimensional Gaussian test case on the unit square $\Om = (0,1)^2$, for which the solution is given by
\begin{equation}\label{eq:ex_sol_1}
  u_{\text{ex}}(\bm{x}) = u_\text{gauss}(\bm{x}):= -\exp\left(-\frac{\Vert \bm{x} - \bm{\mu}\Vert^2_2}{2\sigma^2}\right),
\end{equation}
where $\sigma^2>0$ represents the standard deviation, and $\bm{\mu}\in \R^2$ is the centre of the distribution. 
We are interested in the Gaussian curvature problem \eqref{eq:guass_curvature} for which in fact $f(\bm{x},\bm{y})$ is non-Lipschitz with respect to $\bm{y}$. Thus, we define the function
\begin{equation}\label{eq:for_newton}
  f_\mathrm{g}(\bm{x}): =\det(\mathrm{D}^2 u_\text{gauss}(\bm{x})) =\left(1 - \frac{\Vert \bm{x} - \bm{\mu}\Vert^2_2 }{\sigma^2}\right)\frac{u_{\text{g}}(\bm{x})^2}{\sigma^4} ,
\end{equation}
where the right hand side function in \eqref{eq:guass_curvature} becomes
\begin{equation*}
    f(\bm{x}) = f_{\rm g}(\bm{x}) \left(1 + \left( \frac{(\bm{x} - \bm{\mu})}{\sigma^2}u_{\text{g}}\right)^2 \right)^2,
\end{equation*}
observing that $\nabla u_{\rm gauss}(\bm{x})= (\bm{x}-\bm{\mu}) u_{\rm gauss}/\sigma^2$.
To ensure that the Monge-Ampère equation is elliptic,  $f_\mathrm{g}\geq 0$ needs to be satisfied. 
This holds, if and only if $\Vert x - \mu\Vert_2 \leq \sigma$, for all $x \in \Om$ which puts constraints on $\mu$, $\sigma$. The specific parameters for our test problem are shown in \Cref{fig:gauss}. We will test two different initial guesses ($C_1$, $C_2>0$)
\begin{subequations}
\begin{align}
  & u_1^0(x,y) := u_{\text{ex}}(x,y) - C_1 x(1-x)y(1-y), \label{eq:initial_condition1} \\
& u_2^0(x,y) := u_{\text{ex}}(x,y) + C_2 x(1-x)y(1-y). \label{eq:initial_condition2}  
\end{align}
\end{subequations}

\begin{figure}[h!]
  \centering
  \begin{subfigure}[c]{0.42\textwidth}
      \centering
      \includegraphics[width=\linewidth]{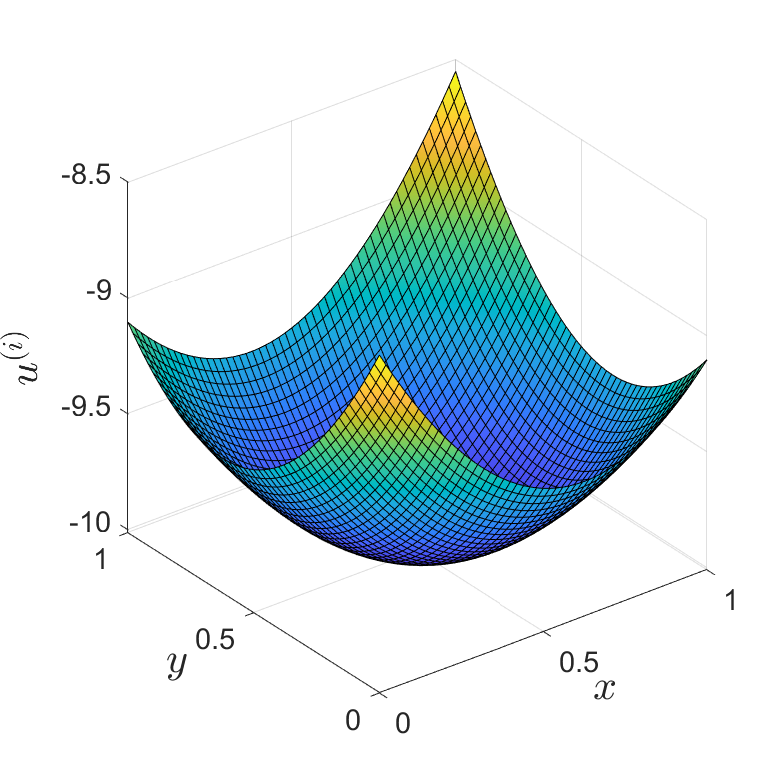}
      \caption{}
  \end{subfigure}
  \hfill
  \begin{subfigure}[c]{0.42\textwidth}
      \centering
      \includegraphics[width=\linewidth]{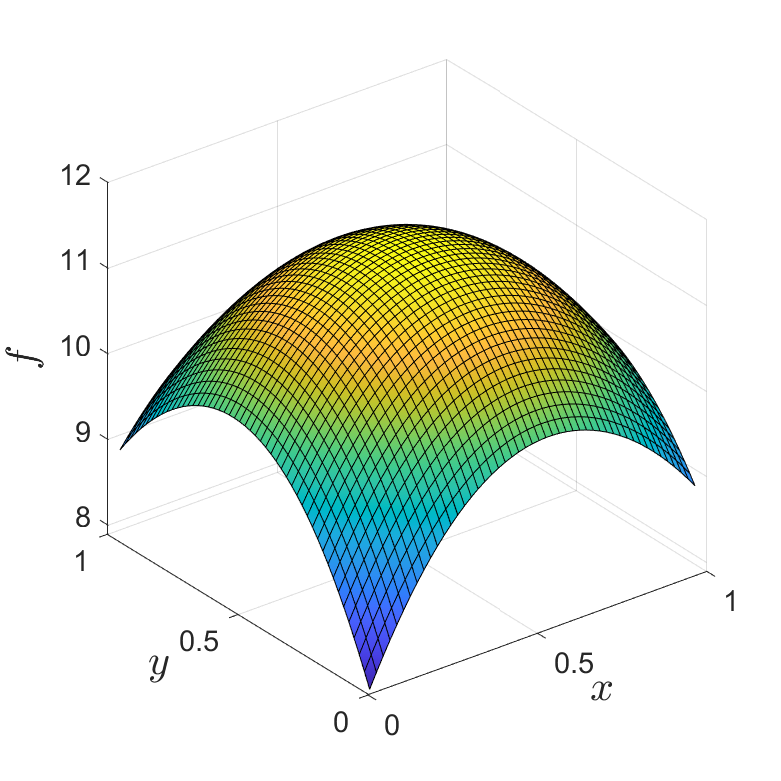}
      \caption{}
  \end{subfigure}
  \caption{The exact solution (a) and right-hand side $f(x)$ (b) for the Gaussian function for $\sigma = 1$ and $\mu=(0.5,0.5)$.}\label{fig:gauss}
\end{figure} 

In our experiments, we choose the constants $C_1 = 30$ and $C_2 = 10$. It is important to note that the initial $u_1^0$ is convex, whereas $u_2^0$ has a saddle shape (see \Cref{fig:saddle_res} (a)) which extends beyond our assumptions in \Cref{thm:convergence,thm:Linfty} of the convexity of iterates.

\subsubsection{Rapidly oscillating solution ($\mu_f=0$)} 
Consider the Monge-Ampère equation \eqref{eq:main_problem} with $\mu_f=0$ for this case. We choose the exact solution to be the Gaussian function given in equation \eqref{eq:ex_sol_1}, perturbed by subtracting a sinus term, resulting in
\begin{equation*}
  u_{\text{ex}}(x,y) = u_{\text{g}}(x,y) - \epsilon_s \sin(l \pi x )\sin(l \pi y ),
\end{equation*}
where $\epsilon_s$ is a constant. To ensure the non-negativity of the right-hand side, it is necessary for $\epsilon_s$ to be sufficiently small.  In this case, we set $\epsilon_s$ to approximately $10^{-3}$. In this example, the right-hand side can become $0$ at several points near the corners. The initial condition is set to $u_1^0$ with the constant $C_1 = 5$.
As shown in \Cref{fig:sinus_comb}, (a) illustrates that $u_\text{ex}$ remains nearly unchanged, while (b) reveals the oscillatory behaviour induced on $f=\det(\mathrm{D}^2 u_{\rm ex})$ by the sinus term. 
\begin{figure}[h!]
  \centering
  \begin{subfigure}[b]{0.44\textwidth}
      \centering
      \includegraphics[width=\linewidth]{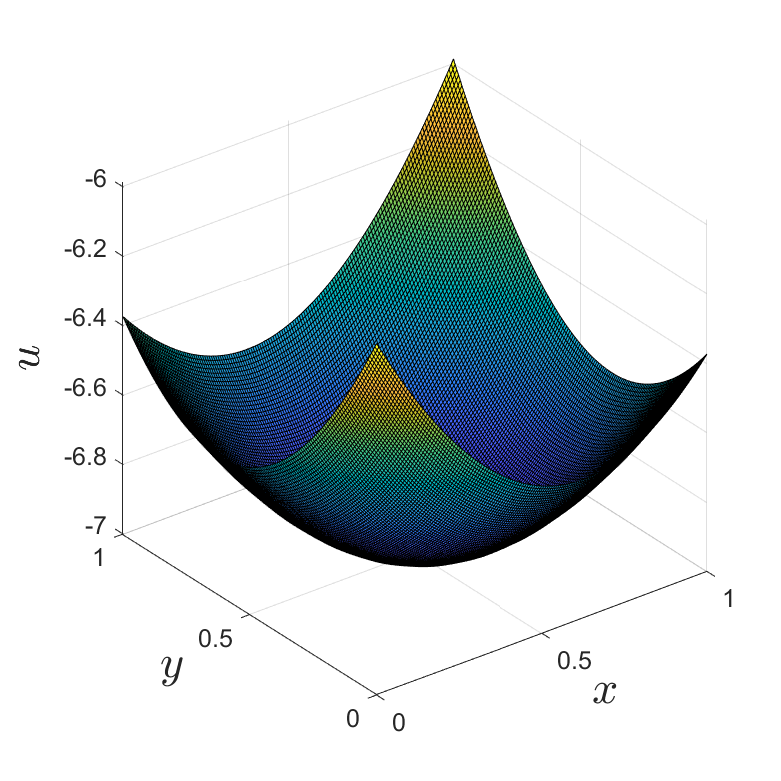}
      \caption{}
  \end{subfigure}
  \hfill
  \begin{subfigure}[b]{0.44\textwidth}
      \centering
      \includegraphics[width=\linewidth]{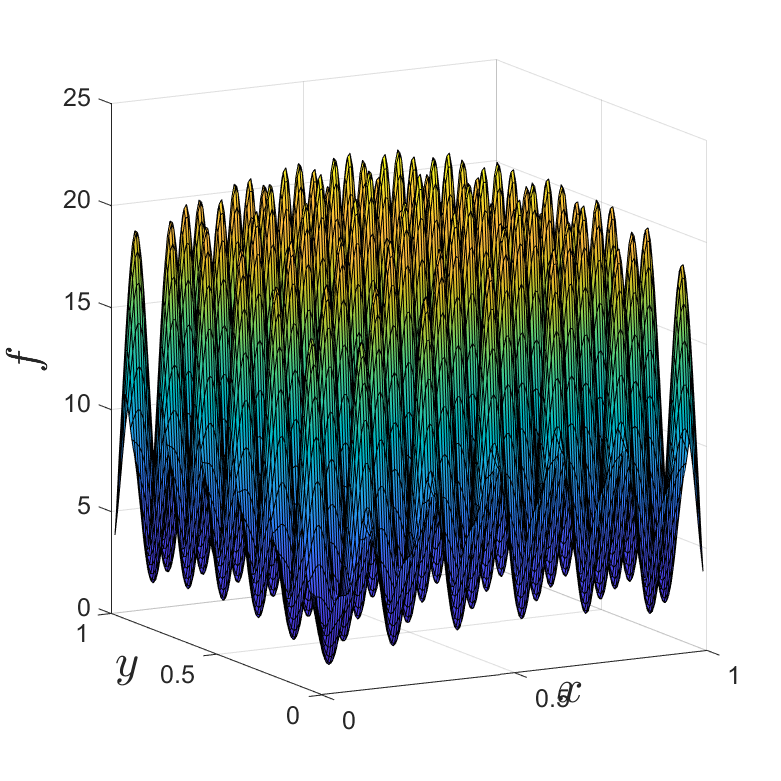}
      \caption{}
  \end{subfigure}
  \caption{The exact solution (a) and right-hand side (b) for the rapidly oscillating solution case. Here, $\sigma = 1$, $\mu=(0.5,0.5)$, and $l=12$.}
  \label{fig:sinus_comb}
\end{figure}

For numerical experiments, the maximum number of iterations $i_\text{max}$ is set to 1500 and the tolerance $\delta_\text{tol}$ is fixed at $10^{-16}$. In the subsequent analysis, we compare several computational results. 
The experiments are conducted using Python. The Python implementation uses the \texttt{scipy} and \texttt{pyamg} \cite{pyamg2023} libraries to solve the linear system of equations.
All computations are performed on an Intel Core i7-12700H under similar operating conditions.

\subsection{Comparison of different implementations of the L-Scheme}\label{subsec:compL}
Comparison of various implementations of the L-Scheme will be conducted with the initial guess $u_1^0$. 
In the context of the Green's function approach, a truncation parameter of $M= 50$ in \eqref{eq:green_trapez} has been selected for the analysis. 
The choice of this parameter represents a balance between computational efficiency and accuracy. 
A smaller truncation parameter typically results in faster convergence, as fewer terms are included in the series expansion, thereby reducing the computational cost. 
This may come at the expense of precision in the obtained results. 
Thus, the selection of the truncation parameter is a crucial aspect of optimizing the trade-off between convergence speed and result fidelity in the Green's function method. 

Using Python, we compare different strategies for the finite difference discretisation of the Laplacian, utilising different preconditioners. 
 The preconditioning of the matrix is incorporated into the timing of the methods.

First we verify the convergence of the L-Scheme for the Gaussian test case. The L-Scheme converges for both the convex and the saddle-shaped initial guess as illustrated in \Cref{fig:saddle_res} (a). This is despite the fact that the saddle-shaped initial guess in  \Cref{fig:saddle_res} (b) violates the convexity assumption of iterates in \Cref{thm:convergence}. All variations of the L-Scheme converge after the same number of iterations, as shown in \Cref{fig:saddle_res} (a). 

\begin{figure}[h!]
  \centering
  \begin{subfigure}[c]{0.48\textwidth}
      \centering
      \includegraphics[width=\linewidth]{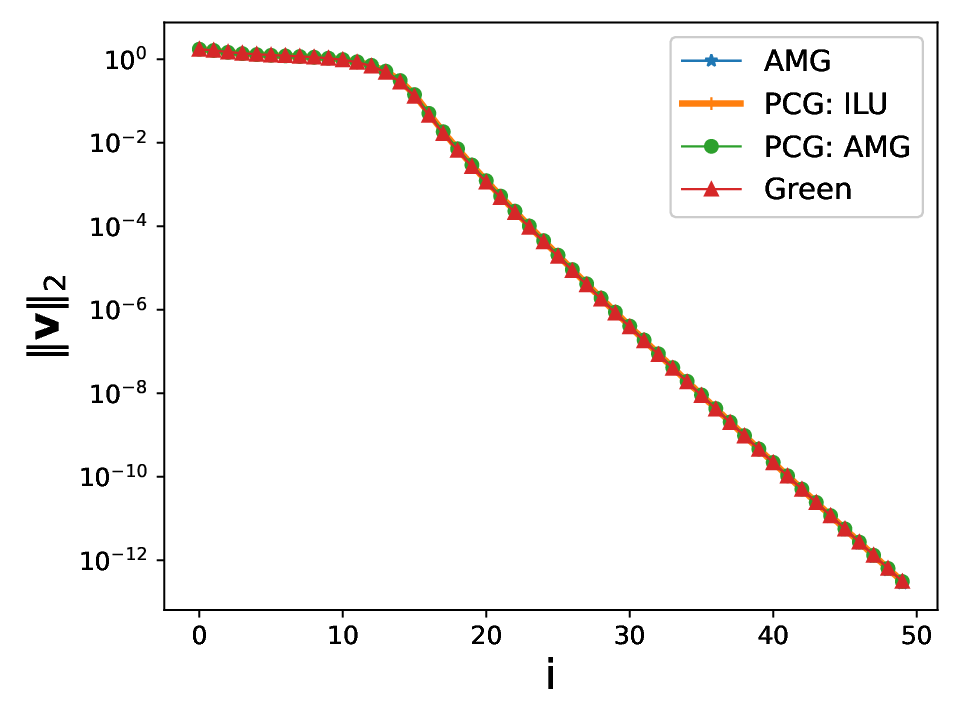}
      \caption{}
  \end{subfigure}
  \hfill
    \begin{subfigure}[c]{0.42\textwidth}
      \centering
      \includegraphics[width=\linewidth]{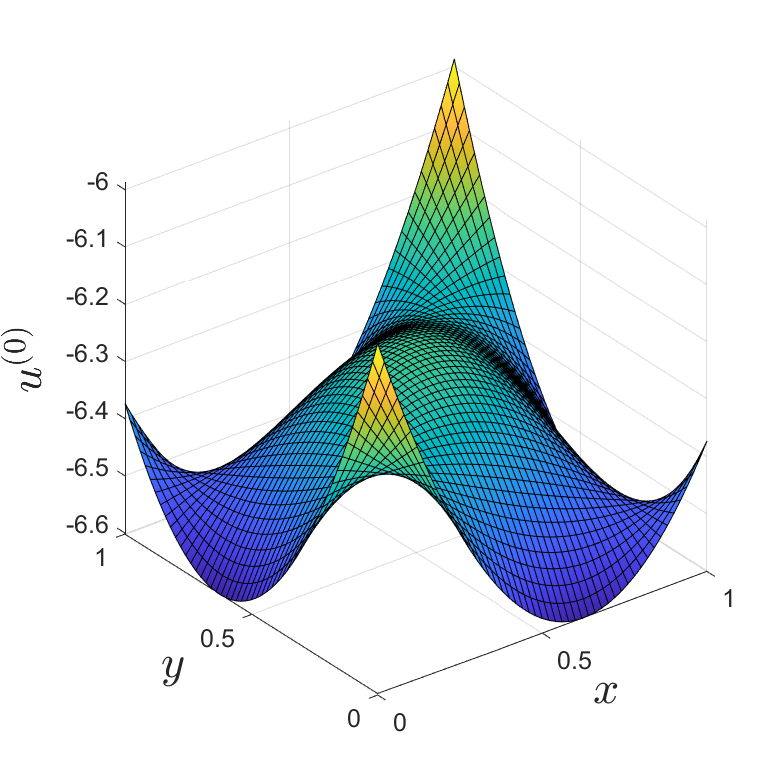}
      \caption{}
  \end{subfigure}
  \caption{[\Cref{subsec:compL}] Convergence of the L-Scheme implementations for the saddle-shaped initial guess (a), where $i$ is the iteration index and \textbf{v} is the update $u^{i+1} - u^{i}$, and the saddle-shaped initial guess $u_2^0$ (b).}\label{fig:saddle_res}
\end{figure}

Next, we compare the different implementations of the Gaussian test problem in \Cref{fig:l_scheme_gaussian}. Important to notice in \Cref{fig:l_scheme_gaussian} (a) is that the Green's function approach is slower in comparison to the other approaches, and we do not obtain any results if $N > 100$. This is because the matrix $\underline{\bm{G}}_{\rm P}\in \R^{N^2\times N^2}$  involved in Green's function method \eqref{eq:GP_matrix} is a full matrix. This restricts the applicability of this approach to coarse meshes as the storage and computational requirements for the resulting matrices become prohibitively large for finer meshes (large $N$). 

\begin{figure}[h!]
  \centering
  \begin{subfigure}[c]{0.48\textwidth}
      \centering
      \includegraphics[width=\linewidth]{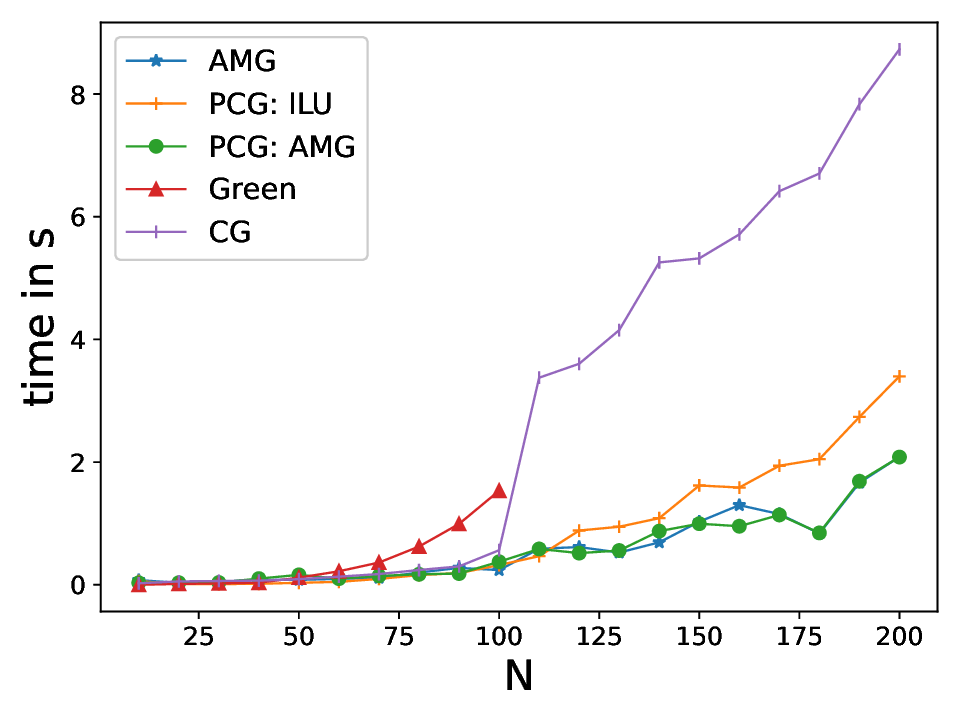}
      \caption{}
  \end{subfigure}
  \hfill
  \begin{subfigure}[c]{0.48\textwidth}
      \centering
      \includegraphics[width=\linewidth]{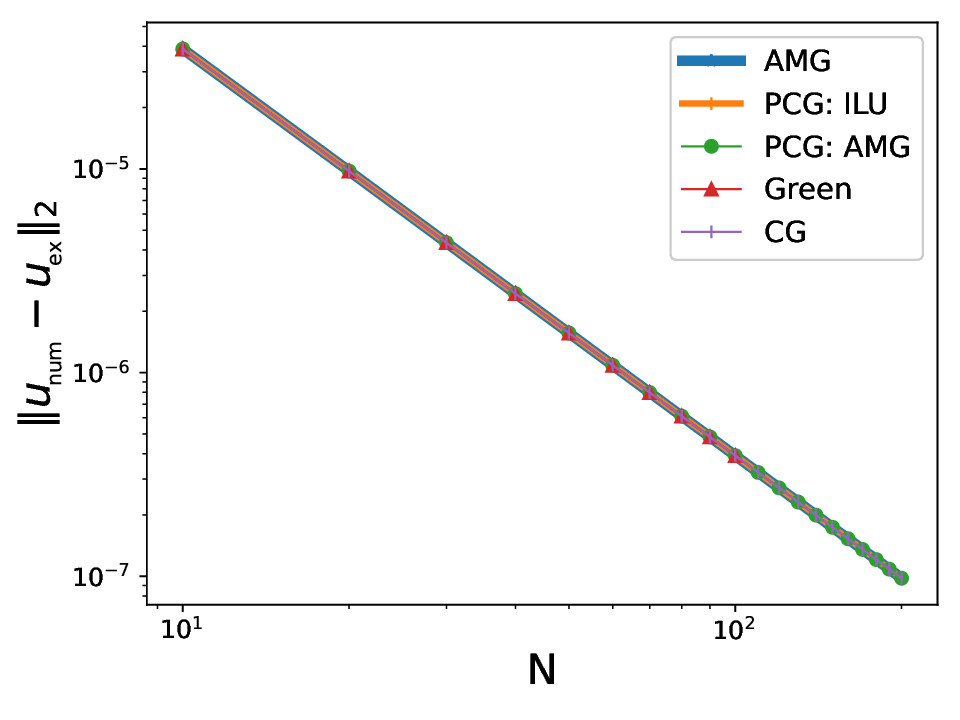}
      \caption{}
  \end{subfigure}
  \caption{[\Cref{subsec:compL}] CPU time (a) up to convergence, and final discretisation error (b) of the Gaussian test problem (L-Scheme) against mesh divisions $N$. }
  \label{fig:l_scheme_gaussian}
\end{figure}

Conversely, the finite difference matrices remain sparse, and are thus, much more scalable.  
The preconditioned gradient algorithm exhibits reduced computational time when employing a multigrid preconditioner compared to an incomplete LU preconditioner. The reason behind that is, that the conjugate gradient with AMG preconditioning needs less iterations than with incomplete LU.  Notably, the performance differences between the PCG method with AMG preconditioning and directly solving the system using AMG are marginal. In \Cref{tab:inner_ieters} we can observe that the number of inner iterations for the conjugate gradient solver are different if we use preconditioners. The number of fixed-point iterations is constant, but the number of inner iterations for the conjugate gradient algorithm is significantly larger if we do not apply preconditioners. This explains why the conjugate gradient method is significantly slower which can be seen in \Cref{fig:l_scheme_gaussian} (a).
As shown in \Cref{fig:l_scheme_gaussian} (b), all algorithms converge to the same solution across  implementations for a fixed discretisation, illustrated by the fact that they have the same discretisation error upon convergence. 

\begin{table}[h!]
\centering
\begin{tabular}{c|cc}
N   & \begin{tabular}[c]{@{}l@{}}with \\ preconditioning\end{tabular} & \begin{tabular}[c]{@{}l@{}}without \\ preconditioning\end{tabular} \\ \hline
50  & 9                                                               & 45                                                                 \\
100 & 15                                                              & 109                                                                \\
200 & 27                                                              & 249                                                               
\end{tabular}
\caption{[\Cref{subsec:compL}] Average number of conjugate gradient iterations required for convergence at different mesh sizes, with and without preconditioning for the Gaussian test case}.
\label{tab:inner_ieters}
\end{table}

Examining the results presented in \Cref{fig:combined_10_results} for the rapidly oscillating test case, we observe that although the problem becomes more difficult to solve, the computation times do not increase compared to the Gaussian test case. 
In contrast, the $2$-norm of the error $\Vert u_{\!_{\mathrm{num}}} - u_{\!_{\mathrm{ex}}}\Vert_{2}$, which represents the discretisation error of the converged solution, is larger than that observed for the Gaussian example.
This is expected, as the right-hand side $f$ exhibits oscillatory behaviour which induces larger discretisation errors due to the multiscale nature.
In \Cref{fig:combined_10_results}, it has been demonstrated that the behaviour of the various solvers is analogous to that of the Gaussian example. It is evident that the Green's function approach is the most time-consuming.
As the value of $l$ decreases, the computational time required is reduced. 
If $l$ increases, the problem is more difficult to solve, therefore the number of iterations and the computation time increases, which can be seen in \Cref{fig:combined_10_results}. 

\begin{figure}[h!]
  \centering
  \begin{subfigure}[c]{0.48\textwidth}
      \centering
      \includegraphics[width=\linewidth]{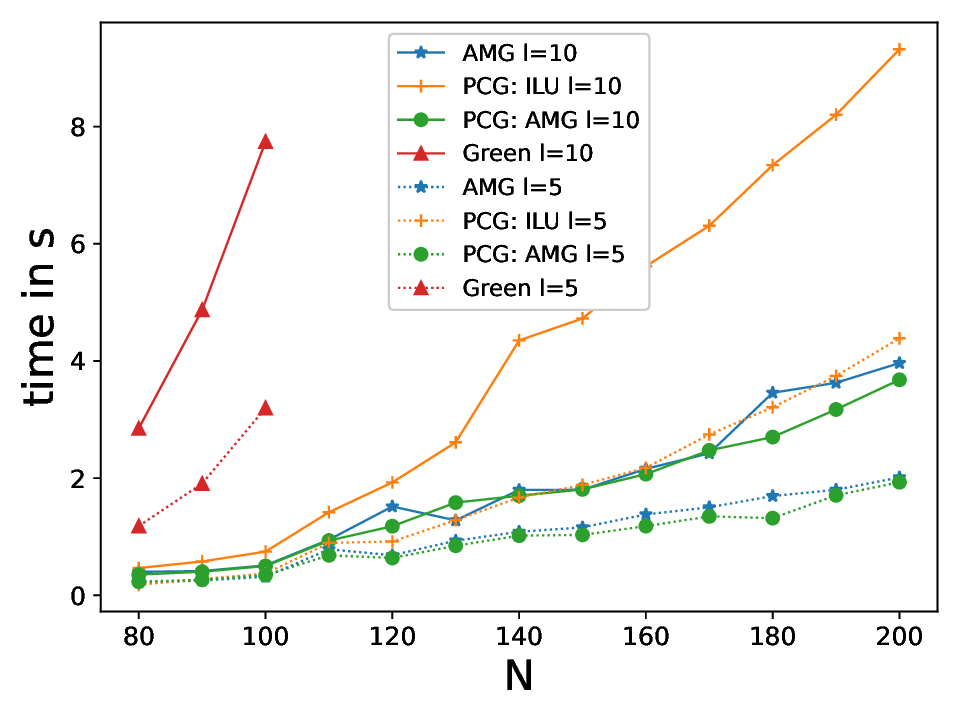}
      \caption{}
  \end{subfigure}
  \hfill
  \begin{subfigure}[c]{0.48\textwidth}
      \centering
      \includegraphics[width=\linewidth]{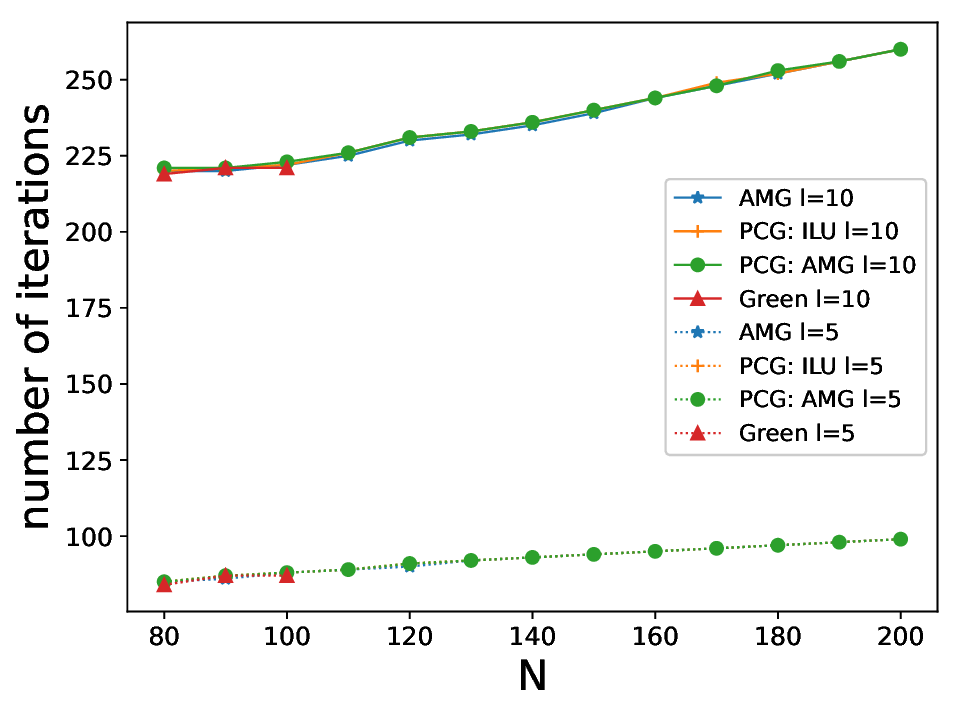}
      \caption{}
  \end{subfigure}
  \caption{[\Cref{subsec:compL}] Results of the rapidly-oscillating test problem: CPU time (a) and iterations (b) required for convergence for two different $l$-values against mesh divisions $N$.}
  \label{fig:combined_10_results}
\end{figure}


\subsection{Comparison with Newton iteration}\label{subsec:New}
For the comparison with Newton iteration we use the Gaussian test problem with $\mu_f =0$, i.e., the exact solution is given by \eqref{eq:ex_sol_1} and the right hand side is given by \eqref{eq:for_newton}. In other words, the right hand side is independent of $\nabla u$, which is necessary because Newton does not converge otherwise even for small initial guesses.  
In this section, we focus on comparing three different approaches: The L-Scheme with Green's function and the version with the preconditioned gradient solver will be compared against the Newton algorithm. For comparison with Newton, we are only looking at the Gaussian test case \eqref{eq:ex_sol_1} with convex initial guess \eqref{eq:initial_condition1}. First, we examine convergence behaviour of the various schemes for a mesh size $N=30$. For the Newton method, the constant $C_1$ in equation \eqref{eq:initial_condition1} must be set to $0.1$, which is close to the exact solution. Larger $C_1$ values lead to divergence. We have seen in \Cref{{subsec:compL}} that the initial guess for the L-Scheme can be much further away from the exact solution, e.g. $C_1=30$, the scheme still converges. It even converges for non-convex initial guesses ($C_2=10$ in \eqref{eq:initial_condition2}) which Newton does not. For the case when it does converge ($C_1=0.1$),  as shown in \Cref{fig:iter_vs_err} (a), Newton's method requires fewer iterations to meet the convergence criteria. This observation aligns with theoretical expectations. On the other hand, both the L-Schemes exhibit roughly the same convergence behaviour.

\begin{figure}[h!]
  \centering
  \begin{subfigure}[b]{0.48\textwidth}
      \centering
    \includegraphics[width=\linewidth]{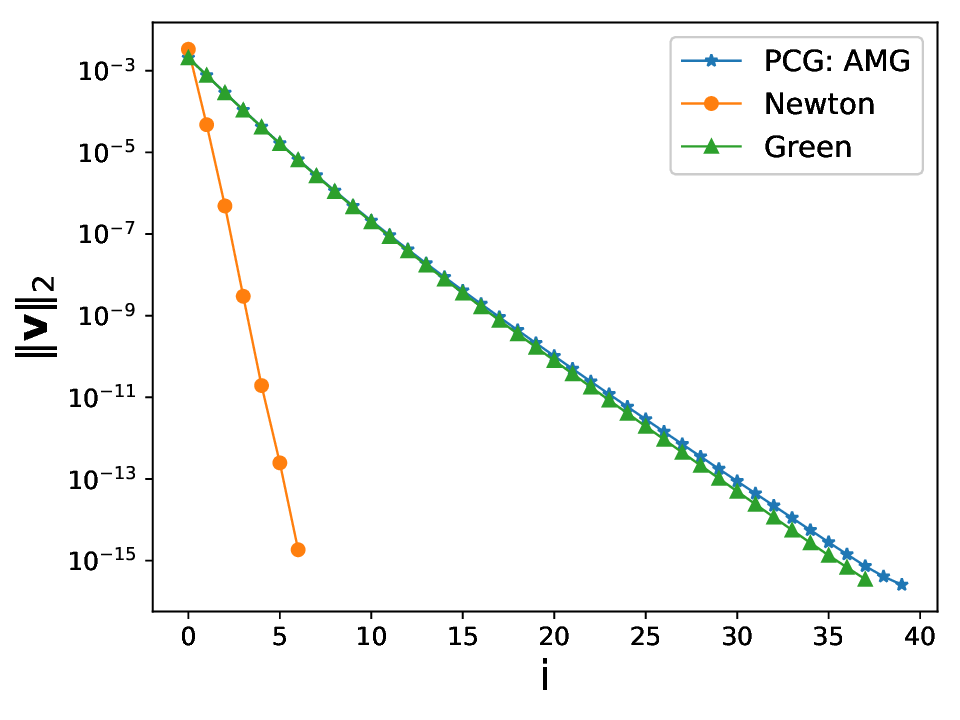}
  \end{subfigure}
  \hfill
  \caption{[\Cref{subsec:New}] Convergence of the three linearisation schemes (L-Scheme Green's function and finite difference versions, and Newton method) for $N=30$ and the Gaussian test problem, where $i$ is the iteration index and \textbf{v} is the update $u^{i+1} - u^{i}$.}
  \label{fig:iter_vs_err}
\end{figure}
In the following stage of our analysis, we will examine the behaviour of the various schemes in relation to different mesh sizes. First, we inspect how Newton iteration behaves with different preconditioners to keep the comparison with the L-Scheme fair. 
A comparison between the computation times in \Cref{fig:newton_gaussian} (a) and the number of iterations required in \Cref{fig:newton_gaussian} (b) reveals that the corresponding curves exhibit a similar behaviour. Furthermore, all considered approaches demonstrate comparable performance except for $N=90$, where the version without preconditioning is slower. 
Since the preconditioned conjugate gradient solver with algebraic multigrid preconditioning seems to perform the best for the regular L-Scheme, we are using this linear solver for Newton iteration as well. 
In the context of Newton's method, the calculation of the Jacobian matrix in each step is a computationally expensive process. Given that the preconditioning of the matrix must be performed in every iteration, the efficiency of this approach is comparable to that of a direct solver, offering minimal gains in efficiency.   For fine discretisations, for example $N \geq 90$, Newton diverges, which is another drawback of the algorithm.

\begin{figure}[htb!]
  \centering
  \begin{subfigure}[b]{0.48\textwidth}
      \centering
      \includegraphics[width=\linewidth]{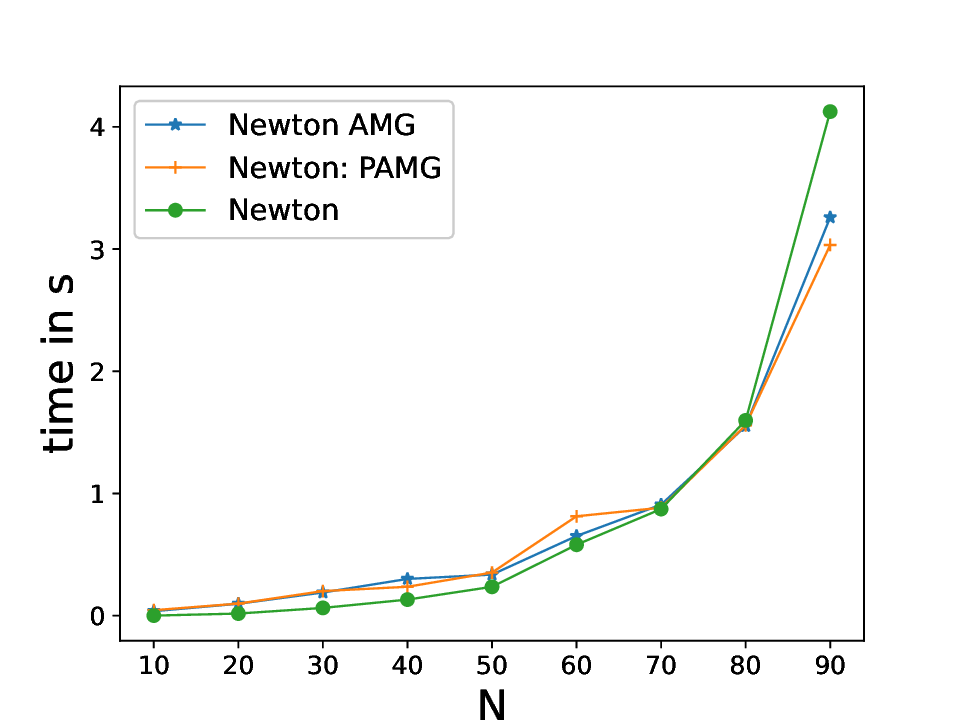}
      \caption{}
  \end{subfigure}
  \hfill
  \begin{subfigure}[b]{0.48\textwidth}
      \centering
      \includegraphics[width=\linewidth]{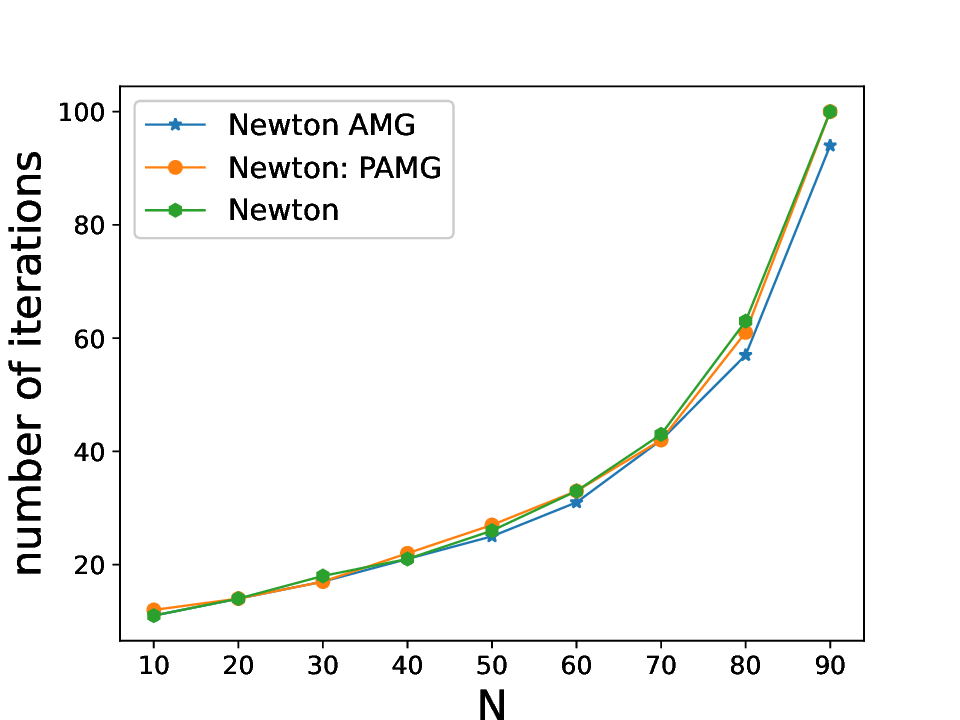}
      \caption{}
  \end{subfigure}
  \caption{[\Cref{subsec:New}] Mesh-study of the Gaussian test problem of the Newton scheme with different preconditioners: computation time (a) and number of iterations (b) }
  \label{fig:newton_gaussian}
\end{figure} 

Finally, we compare the L-Scheme and Newton iteration in \Cref{fig:results_comp}.
From \Cref{fig:results_comp} (a), it is evident that all the algorithms demonstrate similar performances for very coarse meshes. 
However, as the mesh size becomes finer, the performance of the Newtons method noticeably deteriorates, exhibiting slower convergence compared to the other algorithms. 
The Green's function approach also experiences a decline in performance for finer meshes, although to a lesser extent.
In contrast, the L-Scheme with preconditioned conjugate gradient and a finite difference discretisation for the Laplacian stands out as the most robust and efficient algorithm across all mesh sizes. Not only does it outperform the other two methods in terms of speed, but it also maintains its efficiency even for very fine meshes. 
This makes the L-Scheme particularly advantageous for applications requiring high-resolution computations.
\begin{figure}[htb!]  
  \centering
    \begin{subfigure}[b]{0.48\textwidth}
        \centering
        \includegraphics[width=\linewidth]{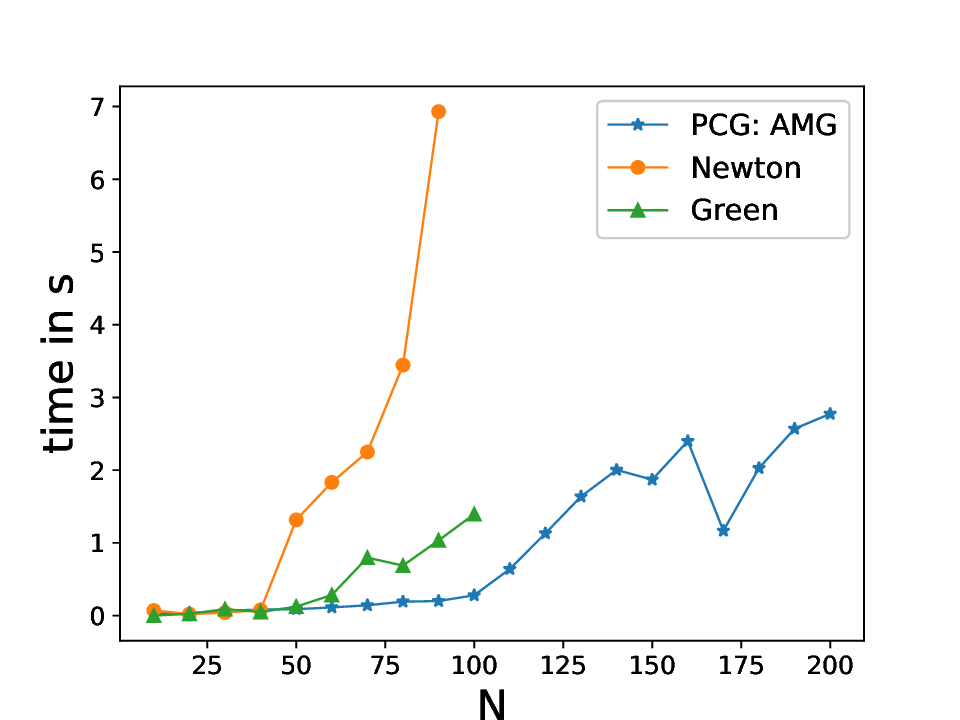}
        \caption{}
    \end{subfigure}
    \hfill
    \begin{subfigure}[b]{0.48\textwidth}
        \centering
        \includegraphics[width=0.92\linewidth]{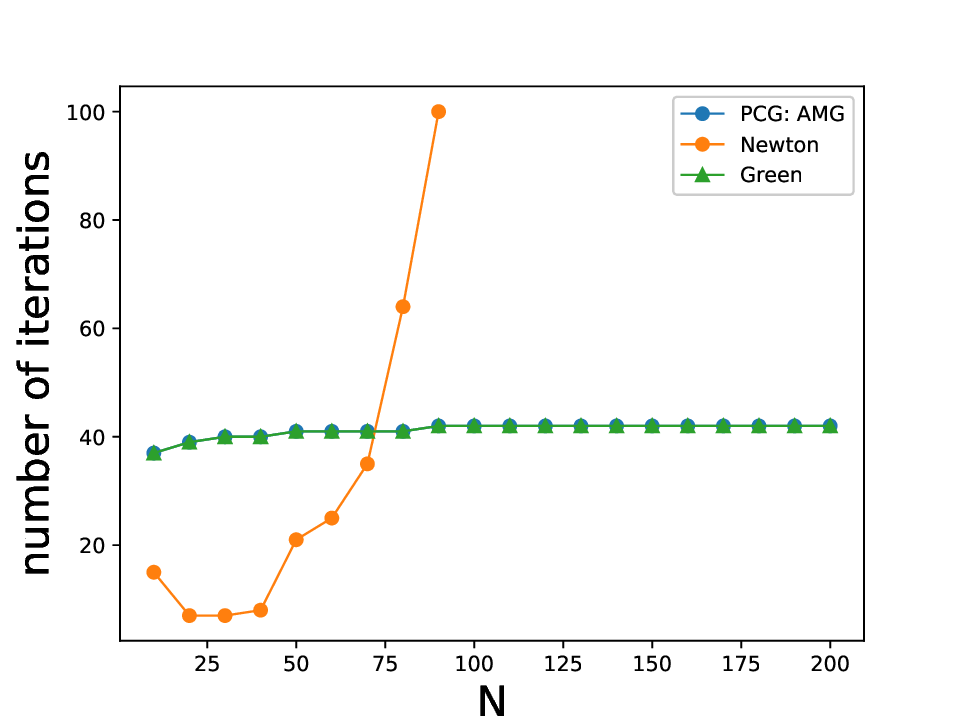}
        \caption{}
    \end{subfigure}
    \caption{[\Cref{subsec:New}] Results for the Gaussian test problem, computation time (a) and number of iterations (b) for L-Scheme Green's function and finite difference versions, and Newton method.}
    \label{fig:results_comp}
  \end{figure}

In \Cref{fig:results_comp} (b) we observe distinct patterns in the convergence behaviour of the iterative schemes with respect to mesh sizes too. Newton method generally requires fewer iterations to converge to the solution for mesh sizes with $N<75$, but the number of iterations increases drastically as the mesh is refined. 
On the other hand, the L-Scheme demonstrates a nearly constant iteration count for different values of $N$. 
This consistency supports that the L-Scheme's performance is hardly to variations in mesh size, making it a reliable choice for a wide range of mesh resolutions.

\section{Conclusions} \label{sec: concl and disc}

In this work, we have introduced and analysed a robust and fast iterative scheme for solving the elliptic Monge-Ampère equation with Dirichlet boundary conditions. While various numerical strategies have been proposed to address the challenges posed by the Monge-Ampère equation, particularly through different iterative and discretisation techniques, our approach distinguishes itself by employing a variant, that linearises the equation and uses a fixed-point iteration, which leads to the solution of a Poisson problem at each step. The right-hand side is constructed using a weighted residual. Well-posedness and consistency of the iteration are shown \Cref{thm:exists}, and convergence is proven in both $H^2$ and $L^\infty$ for the classical and generalised solutions respectively in \Cref{thm:convergence,thm:Linfty}, provided the weighted constant is larger than a power of the largest eigenvalue of the Hessian matrix and a convexity constraint is satisfied.

 The algorithm's robustness with respect to nonlinearity, degeneracy, and oscillations is a key strength of the approach. Since the iterative method effectively solves a Poisson problem in each step, the convergence is expedited by using any method suitable for accelerating its solving.  We pursue two such strategies here based on finite difference discretisation: fixed preconditioners and Green's function.

Extensive numerical experiments, \Cref{sec:numRes}, confirm the theoretical results and further highlight the practical advantages of the L-scheme. When combined with an appropriate preconditioner e.g. algebraic multigrid, the method demonstrates exceptional speed and robustness, even on fine grids. Compared to Newton’s method (\Cref{subsec:New}), which suffers from sensitivity to initial guesses and high computational cost due to repeated Jacobian evaluations, the L-scheme offers a stable and efficient alternative. Notably, the number of iterations remains essentially constant regardless of mesh refinement, emphasising the method's grid-independence (\Cref{subsec:compL}).


\renewcommand{\thesection}{\Alph{section}}
\renewcommand{\thesubsection}{\thesection.\arabic{subsection}}
\renewcommand{\thesubsubsection}{\thesubsection.\arabic{subsubsection}}

\setcounter{section}{0}
\section{Appendix} 

\subsection{Elliptic regularity}
\begin{lemma}[Elliptic Regularity]\label{lem:lemma_for_proof}
    Let $u \in H_0^1(\Om) \cap H^2(\Om)$. Then, there exists a constant $C_E \geq 1,$ independent of $u$, such that
    \begin{equation}\label{eq:lem_to_proof}
        \Vert u \Vert_{H_2} \leq C_E \Vert \D u \Vert_{L^2(\Om)}.
    \end{equation}
\end{lemma}
\noindent \textit{Proof:} Define $g \in L^2(\Om) $ through the Poisson equation
\begin{equation*}
    -\D u =: g.
\end{equation*}
Then, $u\in H^1_0(\Om)$ is also a weak solution of this Poisson problem. By definition of a weak solution, $u$ satisfies for all $v \in H_0^1(\Om)$
\begin{equation*}
    \langle \nabla  u, \nabla v \rangle_{L^2(\Om)} = \langle g, v \rangle_{L^2(\Om)}.
\end{equation*}
Considering the special case $u = v$, we find
\begin{equation*}
    \Vert \nabla u \Vert_{L^2(\Om)}^2 := \langle \nabla u, \nabla u \rangle_{L^2(\Om)} = \langle g, u\rangle_{L^2(\Om)} \geq 0.
\end{equation*}
With the Cauchy-Schwarz inequality, this becomes
\begin{equation}
      \Vert \nabla u \Vert_{L^2(\Om)}^2 \leq \Vert g \Vert_{L^2(\Om)} \Vert u\Vert_{L^2(\Om)} := \Vert \D u \Vert_{L^2(\Om)} \Vert u\Vert_{L^2(\Om)}.
\end{equation}
Applying the Poincaré inequality, there exists a $C_{\mathrm{P},\Om} > 0$ such that the equation above satisfies
\begin{equation}
    \Vert \nabla u \Vert_{L^2(\Om)}^2 \leq  C_{\mathrm{P},\Om} 
 \Vert \D u \Vert_{L^2(\Om)} \Vert \nabla u\Vert_{L^2(\Om)}, \;\;\text{ or }\;\; \Vert \nabla u \Vert_{L^2(\Om)} \leq C_{\mathrm{P},\Om} \Vert \D u \Vert_{L^2(\Om)} \label{eq:c1}.
\end{equation}
We use the elliptic regularity theorem, c.f. Evans \cite[Chapter 6]{evans2022partial}, which implies that there exists a constant $C_1 > 0$ such that
\begin{equation*}
    \Vert u \Vert_{H^2(\Om)} \leq C_1 (\Vert u \Vert_{L^2(\Om)} + \Vert g \Vert_{L^2(\Om)}) := C_1 (\Vert u \Vert_{L^2(\Om)} + \Vert \D u \Vert_{L^2(\Om)}).
\end{equation*}
By applying Poincaré's inequality to this equation, we conclude that there exists a $C_2 >0$ such that
\begin{equation}\label{eq:c3}
     \Vert u \Vert_{H^2(\Om)} \leq C_2 (\Vert  \nabla u \Vert_{L^2(\Om)} + \Vert \D u \Vert_{L^2(\Om)}).
\end{equation}
Combining \eqref{eq:c1} and \eqref{eq:c3} gives us the desired regularity result \eqref{eq:lem_to_proof}. 
\hfill \(\square\)

\bibliographystyle{ieeetr}

\begin{thebibliography}{10}

\bibitem{de2014monge}
G.~De~Philippis and A.~Figalli, ``The {M}onge--{A}mp{\`e}re equation and its
  link to optimal transportation,'' {\em Bulletin of the American Mathematical
  Society}, vol.~51, no.~4, pp.~527--580, 2014.

\bibitem{glimm2003optical}
T.~Glimm and V.~Oliker, ``Optical design of single reflector systems and the
  {M}onge--{K}antorovich mass transfer problem,'' {\em Journal of Mathematical
  Sciences}, vol.~117, no.~3, pp.~4096--4108, 2003.

\bibitem{ten2025inverse}
J.~H.~M.~ten Thije~Boonkkamp, K.~Mitra, M.~J.~H.~Anthonissen, L.~Kusch, P.~Braam, and
  W.~L.~IJzerman, ``Inverse freeform design in non-imaging optics: {H}amilton’s
  theory of geometrical optics, optimal transport, and least-squares solvers,''
  {\em Frontiers in Physics}, vol.~13, p.~1518660, 2025.
  
  \bibitem{trudinger2008monge}
N.~S. Trudinger and X.-J. Wang, ``The monge-ampere equation and its geometric
  applications,'' {\em Handbook of geometric analysis}, vol.~1, pp.~467--524,
  2008.

\bibitem{santambrogio2015optimal}
F.~Santambrogio, {\em Optimal transport for applied mathematicians}, vol.~87.
\newblock Springer, 2015.

\bibitem{Prins}
C.~R. Prins, R.~Beltman, J.~H.~M. ten Thije~Boonkkamp, W.~L. IJzerman, and
  T.~W. Tukker, ``A least-squares method for optimal transport using the
  {M}onge-{A}mp\`ere equation,'' {\em SIAM Journal on Scientific Computing},
  vol.~37, no.~6, pp.~B937--B961, 2015.

\bibitem{yadav2019least}
N.~Yadav, L.~Romijn, J.~ten Thije~Boonkkamp, and W.~IJzerman, ``A least-squares
  method for the design of two-reflector optical systems,'' {\em Journal of
  Physics: Photonics}, vol.~1, no.~3, p.~034001, 2019.

\bibitem{Awanou}
G.~Awanou, ``Standard finite elements for the numerical resolution of the
  elliptic {M}onge-{A}mp\`ere equation: classical solutions,'' {\em IMA Journal
  of Numerical Analysis}, vol.~35, no.~3, pp.~1150--1166, 2015.

\bibitem{awanou2017standard}
G.~Awanou, ``Standard finite elements for the numerical resolution of the
  elliptic {M}onge--{A}mp{\`e}re equation: Aleksandrov solutions,'' {\em ESAIM:
  Mathematical Modelling and Numerical Analysis}, vol.~51, no.~2, pp.~707--725,
  2017.

\bibitem{awanou2016standard}
G.~Awanou, ``On standard finite difference discretizations of the elliptic
  {M}onge--{A}mp{\`e}re equation,'' {\em Journal of Scientific Computing},
  vol.~69, no.~2, pp.~892--904, 2016.

\bibitem{glowinski2019finite}
R.~Glowinski, H.~Liu, S.~Leung, and J.~Qian, ``A finite
  element/operator-splitting method for the numerical solution of the two
  dimensional elliptic {M}onge--{A}mp{\`e}re equation,'' {\em Journal of
  Scientific Computing}, vol.~79, pp.~1--47, 2019.

\bibitem{brenner2012finite}
S.~C. Brenner and M.~Neilan, ``Finite element approximations of the three
  dimensional {M}onge-{A}mp{\`e}re equation,'' {\em ESAIM: Mathematical
  Modelling and Numerical Analysis}, vol.~46, no.~5, pp.~979--1001, 2012.

\bibitem{vanishing_moments}
X.~Feng and M.~Neilan, ``Mixed finite element methods for the fully nonlinear
  {M}onge–{A}mpère equation based on the vanishing moment method,'' {\em
  SIAM Journal on Numerical Analysis}, vol.~47, no.~2, pp.~1226--1250, 2009.

\bibitem{Froese}
B.~D. Froese and A.~M. Oberman, ``Convergent finite difference solvers for
  viscosity solutions of the elliptic {M}onge-{A}mp\`ere equation in dimensions
  two and higher,'' {\em SIAM Journal on Numerical Analysis}, vol.~49, no.~4,
  pp.~1692--1714, 2011.

\bibitem{gallistl2023convergence}
D.~Gallistl and N.~Tran, ``Convergence of a regularized finite element
  discretization of the two-dimensional {M}onge-{A}mpère equation,'' {\em
  Mathematics of Computation}, vol.~92, no.~342, pp.~1467--1490, 2023.

\bibitem{gallistl2024stability}
D.~Gallistl and N.~Tran, ``Stability and guaranteed error control of
  approximations to the {M}onge-{A}mpère equation,'' {\em Numerische
  Mathematik}, vol.~156, no.~1, pp.~107--131, 2024.

\bibitem{NYSTROM}
K.~Nystr\"om and M.~Vestberg, ``Solving the {D}irichlet problem for the
  {M}onge-{A}mp\`ere equation using neural networks,'' {\em Journal of
  Computational Mathematics and Data Science}, vol.~8, p.~100080, 2023.

\bibitem{kawecki2018}
E.~Kawecki, O.~Lakkis, and T.~Pryer, ``A finite element method for the
  {M}onge-{A}mp\`ere equation with transport boundary conditions,'' 2018.

\bibitem{froese_Transp}
B.~D. Froese, ``A numerical method for the elliptic {M}onge-{A}mp\`ere equation
  with transport boundary conditions,'' {\em SIAM Journal on Scientific
  Computing}, vol.~34, no.~3, pp.~A1432--A1459, 2012.

\bibitem{romijn2020monge}
L.~B. Romijn, J.~H.~M. ten Thije~Boonkkamp, M.~J.~H. Anthonissen, and W.~L.
  IJzerman, ``An iterative least-squares method for generated jacobian
  equations in freeform optical design,'' {\em SIAM Journal on Scientific
  Computing}, vol.~43, no.~2, pp.~B298--B322, 2021.

\bibitem{hacking2024}
R.~Hacking, L.~Kusch, K.~Mitra, M. J. H.~Anthonissen, and W. L.~IJzerman, ``A neural
  network approach for solving the {M}onge-{A}mp\`ere equation with transport
  boundary condition,'' 2024.

\bibitem{bertens2021numerical}
M.~W. Bertens, E.~M. Vugts, M.~J.~H. Anthonissen, J.~H.~M. Boonkkamp, and W.~L.
  IJzerman, ``Numerical methods for the hyperbolic {M}onge-{A}mpère equation
  based on the method of characteristics,'' {\em arXiv preprint
  arXiv:2104.11659}, 2021.

\bibitem{bertens2023iterative}
M.~W. Bertens, M.~J.~H. Anthonissen, J.~H.~M. Boonkkamp, and W.~L. IJzerman, ``An
  iterative least-squares method for the hyperbolic {M}onge-{A}mpère equation
  with transport boundary condition,'' {\em arXiv preprint arXiv:2303.15459},
  2023.

\bibitem{POP2004365}
I.~S.~Pop, F.~A.~Radu, and P.~Knabner, ``Mixed finite elements for the {R}ichards'
  equation: linearization procedure,'' {\em Journal of Computational and
  Applied Mathematics}, vol.~168, no.~1, pp.~365--373, 2004.
\newblock Selected Papers from the Second International Conference on Advanced
  Computational Methods in Engineering (ACOMEN 2002).

\bibitem{MITRA20191722}
K.~Mitra and I.~S.~Pop, ``A modified {L}-scheme to solve nonlinear diffusion
  problems,'' {\em Computers \& Mathematics with Applications}, vol.~77, no.~6,
  pp.~1722--1738, 2019.
\newblock 7th International Conference on Advanced Computational Methods in
  Engineering (ACOMEN 2017).

\bibitem{stokke2023adaptive}
J.~Stokke, K.~Mitra, E.~Storvik, J.~Both, and F.~Radu, ``{An adaptive solution
  strategy for Richards' equation},'' {\em Computers \& Mathematics with
  Applications}, vol.~152, pp.~155--167, 2023.

\bibitem{javed2025robust}
A.~Javed, K.~Mitra, and I.~Pop, ``Robust, fast, and adaptive splitting schemes
  for nonlinear doubly-degenerate diffusion equations,'' {\em arXiv preprint
  arXiv:2508.07420}, 2025.


\bibitem{evans2022partial}
L.~Evans, {\em Partial {D}ifferential {E}quations}.
\newblock American Mathematical Society, 2022.

\bibitem{gutierrez2001monge}
C.~Guti{\'e}rrez and H.~Brezis, {\em The Monge-Ampere equation}, vol.~44.
\newblock Springer, 2001.

\bibitem{guan1998dirichlet}
B.~Guan, ``The {D}irichlet problem for {M}onge-{A}mpere equations in non-convex
  domains and spacelike hypersurfaces of constant {G}auss curvature,'' {\em
  Transactions of the American Mathematical Society}, vol.~350, no.~12,
  pp.~4955--4971, 1998.

\bibitem{gilbarg1977elliptic}
D.~Gilbarg and N.~Trudinger, {\em Elliptic partial differential equations of
  second order}, vol.~224.
\newblock Springer, 1977.

\bibitem{korevaar1983capillary}
N.~Korevaar, ``Capillary surface convexity above convex domains,'' {\em Indiana
  University Mathematics Journal}, vol.~32, no.~1, pp.~73--81, 1983.

\bibitem{nguyen2021lipschitz}
B.~T. Nguyen and P.~D. Khanh, ``Lipschitz continuity of convex functions,''
  {\em Applied Mathematics \& Optimization}, vol.~84, no.~2, pp.~1623--1640,
  2021.

\bibitem{zauderer2011partial}
E.~Zauderer, {\em Partial differential equations of applied mathematics}.
\newblock John Wiley \& Sons, 2011.

\bibitem{demmel1997applied}
J.~W. Demmel, {\em Applied numerical linear algebra}.
\newblock SIAM, 1997.

\bibitem{KAASSCHIETER1988265}
E.~Kaasschieter, ``Preconditioned conjugate gradients for solving singular
  systems,'' {\em Journal of Computational and Applied Mathematics}, vol.~24,
  no.~1, pp.~265--275, 1988.

\bibitem{vanek1996algebraic}
P.~Vanek, J.~Mandel, and M.~Brezina, ``Algebraic multigrid by smoothed
  aggregation for second and fourth order elliptic problems,'' {\em Computing},
  vol.~56, no.~3, pp.~179--196, 1996.

\bibitem{van2001convergence}
P.~Van{\v{e}}k, M.~Brezina, and J.~Mandel, ``Convergence of algebraic multigrid
  based on smoothed aggregation,'' {\em Numerische Mathematik}, vol.~88,
  pp.~559--579, 2001.

\bibitem{pyamg2023}
N.~Bell, L.~N. Olson, J.~Schroder, and B.~Southworth, ``{PyAMG}: Algebraic
  multigrid solvers in {P}ython,'' {\em Journal of Open Source Software},
  vol.~8, no.~87, p.~5495, 2023.
  

  

\end{thebibliography}

\end{document}